\documentclass[12pt]{article}
\usepackage{amssymb}
\usepackage{amsmath}
\usepackage{amsthm}
\usepackage{color}
\usepackage{comment}
\usepackage{geometry}
\usepackage{graphicx}

\makeatletter
	\@addtoreset{equation}{section}
\makeatother

\let\originalleft\left
\let\originalright\right
\renewcommand{\left}{\mathopen{}\mathclose\bgroup\originalleft}
\renewcommand{\right}{\aftergroup\egroup\originalright}

\begin{document}

\newcommand{\bO}{{\bf 0}}
\newcommand\cA{\mathcal{A}}
\newcommand\cF{\mathcal{F}}
\newcommand\cL{\mathcal{L}}
\newcommand\cN{\mathcal{N}}
\newcommand\cO{\mathcal{O}}
\newcommand\cT{\mathcal{T}}
\newcommand{\ee}{\varepsilon}
\newcommand\rD{\mathrm{D}}
\newcommand\re{\mathrm{e}}
\newcommand\ri{\mathrm{i}}

\newcommand{\myStep}[2]{{\bf Step #1} --- #2\\}

\newtheorem{theorem}{Theorem}[section]
\newtheorem{corollary}[theorem]{Corollary}
\newtheorem{lemma}[theorem]{Lemma}
\newtheorem{proposition}[theorem]{Proposition}

\theoremstyle{definition}
\newtheorem{definition}{Definition}[section]
\newtheorem{example}[definition]{Example}

\theoremstyle{remark}
\newtheorem{remark}{Remark}[section]

\title{Refracting Filippov systems: sliding dynamics without sliding regions.}
\author{
D.J.W.~Simpson\\\\
School of Mathematical and Computational Sciences\\
Massey University\\
Palmerston North, 4410\\
New Zealand
}

\maketitle

\begin{abstract}

This paper develops fundamental mathematical theory for refracting Filippov systems. These are discontinuous ordinary differential equations with solutions defined in the sense of Filippov, and whose first Lie derivatives vary continuously across discontinuity surfaces. Unlike generic Filippov systems, discontinuity surfaces consist only of crossing regions and their boundaries where both adjacent vector fields are tangent to the discontinuity surface. Crossing orbits spiral around invisible-invisible tangency surfaces, and we derive a formula for the attractive or repulsive strength of these surfaces. We prove crossing orbits cannot converge to tangency surfaces in finite time (no Zeno), and that the limiting dynamics consists of Filippov solutions on the tangency surfaces (second-order sliding motion). We derive a vector field that governs this motion, and characterise the stability of equilibria on tangency surfaces. The methodology is applied to a model of a mechanical oscillator with compliant impacts, and a model of ant colony migration. We also relate refracting Filippov systems to second-order sliding mode control, and show that for two-dimensional systems the results reduce to known theory.

\end{abstract}

\section{Introduction}
\label{sec:intro}

Many physical systems switch between distinct modes of operation,
and are well modelled by discontinuous, piecewise-smooth, ordinary differential equations.
The functional form of the equations changes on {\em discontinuity surfaces}
that consist of crossing regions, attracting sliding regions, and repelling sliding regions,
as determined by the relative direction of the vector field on each side of the surface.
When an orbit of the system reaches a sliding region,
it may then evolve along this region (sliding motion)
where it represents an average of the two adjacent modes of operation, Fig.~\ref{fig:A}a.

\begin{figure}
\begin{center}
\includegraphics[width=12cm]{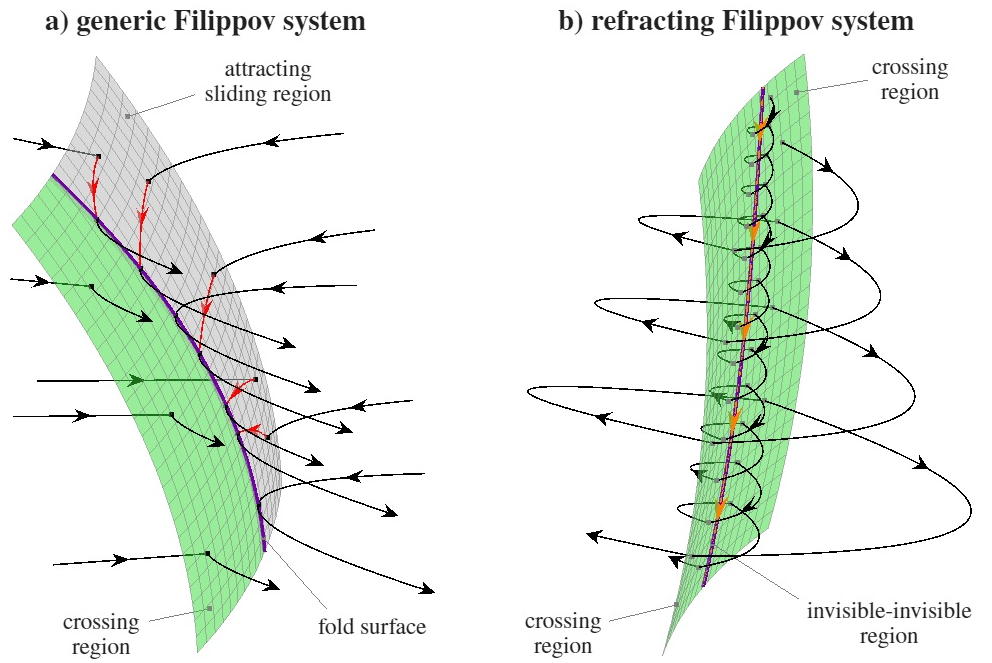}
\caption{
Phase portraits of Filippov systems of the form \eqref{eq:f}.
The discontinuity surface $\Sigma$ is coloured green for crossing regions and grey for attracting sliding regions.
Throughout this paper sliding orbits on sliding regions are coloured red,
as in panel (a), while sliding orbits on tangency surfaces are coloured orange, as in panel (b).
\label{fig:A}
} 
\end{center}
\end{figure}

Discontinuity surfaces may have no sliding regions, as in Fig.~\ref{fig:A}b.
In this example, orbits spiral around the boundary between two crossing regions.
This boundary is where the direction of both vector fields relative to the discontinuity surface change sign.
This is a highly degenerate scenario for piecewise-smooth systems
having no {\em a priori} relationship between the two adjacent vector fields.
However, this scenario arises naturally in models of mechanical systems with compliant impacts \cite{Br99,BlCz99},
and is present in models of diverse applications, including ant colonies \cite{WaQi24},
glacial cycles \cite{PaPa04},
and the progression of diseases subject to intermittent treatment \cite{TaXi12,ViMa20}.

To explain how the situation in Fig.~\ref{fig:A}b may arise in a model, consider a two-piece system
\begin{equation}
\dot{x} = \begin{cases}
f^L(x), & H(x) < 0, \\
f^R(x), & H(x) > 0,
\end{cases}
\label{eq:f}
\end{equation}
with domain $\Omega \subset \mathbb{R}^n$ and discontinuity surface
\begin{equation}
\Sigma = \left\{ x \in \Omega \,\middle|\, H(x) = 0 \right\}.
\label{eq:discSurface}
\end{equation}
Crossing regions are subsets of $\Sigma$
for which the first Lie derivatives
\begin{align*}
\cL_{f^L} H(x) &= \nabla H(x)^{\sf T} f^L(x), &
\cL_{f^R} H(x) &= \nabla H(x)^{\sf T} f^R(x),
\end{align*}
have the same sign, while sliding regions are where the derivatives have opposite sign.
In generic situations, these signs change on two different {\em tangency surfaces}.
But in many models, $f^L$ and $f^R$ share many of the same expressions,
e.g.~one is identical to the another except incorporates a control strategy.
That is, $f^R(x) = f^L(x) + g(x)$, where $g$ is relatively simple.
Then if $\nabla H(x)^{\sf T} g(x) = 0$ throughout $\Sigma$,
then $\cL_{f^L} H(x)$ and $\cL_{f^R} H(x)$ are identical on $\Sigma$,
and so their signs change on the same tangency surface, as in Fig.~\ref{fig:A}b.


With a nod to optics, such systems are termed {\em refracting}.
This nomenclature appears to have been introduced by Teixeira \cite{Te81}
who studied the local structural stability of two-dimensional refracting systems.
This work was generalised to three dimensions by Buzzi {\em et al.}~\cite{BuMe13},
and recently to $n$-dimensions by Siller and Teixeira \cite{SiTe26}.
The results show that a point on a discontinuity surface
having quadratically tangent trajectories for each piece of the system
only generates a structural instability if it is a non-hyperbolic equilibrium.
Several research groups have studied two-dimensional, two-piece, piecewise-linear refracting systems
with the aim of numerating limit cycles and characterising phase portraits \cite{PoRo18,LiLl20,LiLi21b,ShLi21}.
In particular, Li {\em et al}.~\cite{LiLi21} prove that
such systems admit at most one limit cycle.

We use Filippov's convention \cite{Fi88} to define orbits of refracting systems,
and in this way no new rules are required define motion along $\Sigma$.
Such motion necessarily occurs along tangency surfaces.
Since these surfaces are codimension-two,
this motion will be referred to as {\em second-order sliding motion}.

The purpose of this paper is to establish the basic elements for a qualitative theory 
of refracting Filippov systems in any number of dimensions.
It is interesting that applied examples
have not initiated the development of a mathematical framework specific to such systems until now.
Possibly this is because the framework is not essential if a careful analysis of the dynamics is not performed,
or spiralling motion around a tangency surface does not occur.
Also, the framework is not needed if the system is two-dimensional,
for in this case tangency surfaces are points and second-order sliding motion is trivial.
Indeed this work is part of a project to advance the bifurcation theory of piecewise-smooth systems beyond two dimensions.

The remainder of this paper is organised as follows.
Section \ref{sec:F} reviews the basic principles of Filippov systems,
including the sliding vector field, pseudo-equilibria, and tangency points.
Tangency points are classified as visible or invisible
by the signs of the second Lie derivatives, $\cL^2_{f^L} H(x)$ and $\cL^2_{f^R} H(x)$.

In \S\ref{sec:cts} we consider refracting systems
and classify tangency surfaces by the visibility of their points.
We also introduce the second-order sliding vector field $f^T$.
This vector field is derived from Filippov's convention and governs the evolution of orbits on tangency surfaces.
The second-order sliding vector field has a simple functional form because it assumes $\cL^2_{f^L} H(x) \ne \cL^2_{f^R} H(x)$.
Arbitrary order tangencies were treated by Carvalho {\em et al.}~\cite{CaNo24},
and we show that their vector field reduces to ours in the case $\cL^2_{f^L} H(x) \ne \cL^2_{f^R} H(x)$.

For an invisible-invisible tangency surface $\cT$, 
orbits dwell near $\cT$ because there are no visible tangent trajectories to eject orbits from the the vicinity of $\cT$.
Typical orbits spiral around $\cT$, as in Fig.~\ref{fig:A}b,
and in \S\ref{sec:crossingDynamics} we establish the leading-order asymptotics for this motion.
The asymptotics show that orbits move toward $\cT$ if
\begin{equation}
\Lambda = \frac{\cL_{f^L}^3 H}{\big( \cL_{f^L}^2 H \big)^2} - \frac{\cL_{f^R}^3 H}{\big( \cL_{f^R}^2 H \big)^2}
\label{eq:Lambda}
\end{equation}
is negative, and away from $\cT$ if this quantity is positive, Theorem \ref{th:P}.
Geometrically, $\Lambda$ can be intepreted as the difference
between the rate of change of the radius of curvature
of trajectories through $\cT$ for the two pieces of the system
relative to the discontinuity surface.

In \S\ref{sec:noZeno} we prove that spiralling motion cannot reach $\cT$ and transition to second-order sliding motion in finite time, Theorem \ref{th:noZeno}.
This contrasts generic Filippov systems for which orbits reach attracting sliding regions and immediately switch to sliding motion.
It also contrasts hybrid dynamical systems for which infinitely many switches can occur in finite time (the Zeno phenomenon) \cite{JoEg99,NoDa11,ZhJo01}.
Many hybrid models of vibro-impacting mechanical systems exhibit sticking motion along a codimension-two tangency surface,
which is analogous to second-order sliding motion, but can be reached via infinitely switches in finite time \cite{ChVa16,NoPi09}.

By Filippov \cite[Thm.~2, \S 8]{Fi88},
spiralling motion is a continuous perturbation of second-order sliding motion.
That is, second-order sliding motion, as defined through Filippov's convention,
provides the correct way of smoothly averaging the rapidly switching spiralling motion.
This property allows us to use the second-order sliding vector field
to predict how spiralling motion moves about the tangency surface.
In \S\ref{sec:noZeno} we also prove that the size of the continuous perturbation admits a linear bound, Theorem \ref{th:consistency}.

In \S\ref{sec:pseq} we consider equilibria of the second-order sliding vector field.
These are {\em pseudo-equilibria} of the system, and we characterise their stability, Theorem \ref{th:stab2}.
In \S\ref{sec:impacts} and \S\ref{sec:ants} we use the new concepts
to better understand the dynamics of a vibro-impact model, and an ant colony migration model.
In \S\ref{sec:2d} we consider two-dimensional refracting systems.
Here tangency surfaces are points
and we show that the criterion for the stability of such points established in previous works \cite{CoGa01,Si22b}
reduces to the sign of $\Lambda$ in our setting.
In \S\ref{sec:control} we contrast refracting systems with second-order sliding mode control \cite{ShEd14}.
Finally, \S\ref{sec:conc} provides conclusions and an outlook for future work.

\section{Preliminaries}
\label{sec:F}

In this section we define Filippov solutions
and other fundamental aspects of piecewise-smooth ODEs that will be needed in later sections.
For more details on these topics refer to Filippov \cite{Fi88} or the textbooks \cite{DiBu08,Je18b}.

For simplicity, this paper only treats systems of the form \eqref{eq:f}
which admit a single discontinuity surface $\Sigma$.
For systems with more pieces and/or more discontinuity surfaces,
one simply applies the theory given here to subsets of phase space
on which the system has two pieces and one discontinuity surface.
The following definition contains the assumptions we wish to place upon the components of \eqref{eq:f}.

\begin{definition}
A system \eqref{eq:f} is {\em piecewise-$C^k$} ($k \ge 0$) on an open set $\Omega \subset \mathbb{R}^n$ if
$f^L : \Omega \to \mathbb{R}^n$ is $C^k$,
$f^R : \Omega \to \mathbb{R}^n$ is $C^k$, and
$H : \Omega \to \mathbb{R}$ is $C^{k+1}$ with $\nabla H(x) \ne \bO$ for all $x \in \Sigma$.
\end{definition}

In most models, $f^L$, $f^R$, and $H$ are analytic, and the value of $k$ can be taken to be arbitrarily large.
The extra degree of differentiability of $H$ ensures that the sliding vector field (defined below) is $C^k$.
The condition $\nabla H(x) \ne \bO$ ensures that $\Sigma$ is a $C^{k+1}$ codimension-one manifold
by the regular value theorem \cite{Hi76}.

\subsection{Filippov solutions}

Filippov \cite{Fi88} defines solutions for general classes of discontinuous ODEs.
For ODEs of the form \eqref{eq:f}, Filippov's definition can be stated as follows.

\begin{definition}
A {\em Filippov solution} to \eqref{eq:f} is an absolutely continuous function $\phi(t)$
satisfying $\dot{\phi}(t) \in \cF(\phi(t))$ for almost all $t$, where $\cF$ is the set-valued function
\begin{equation}
\cF(x) = \begin{cases}
\left\{ f^L(x) \right\}, & H(x) < 0, \\
\left\{ (1-s) f^L(x) + s f^R(x) \,\middle|\, 0 \le s \le 1 \right\}, & H(x) = 0, \\
\left\{ f^R(x) \right\}, & H(x) > 0.
\end{cases}
\label{eq:cF}
\end{equation}
\end{definition}

Throughout this paper we consider Filippov solutions to \eqref{eq:f},
that is, we treat \eqref{eq:f} as a {\em Filippov system}.

\begin{definition}
An {\em equilibrium} of \eqref{eq:f} is a point $x^* \in \Omega$ for which $\bO \in \cF(x^*)$.
\label{df:eq}
\end{definition}

If $x^* \in \Omega$ is an equilibrium of \eqref{eq:f},
then the constant function $\phi(t) = x^*$ is a Filippov solution to \eqref{eq:f}.
The following definition extends the standard notions of stability to Filippov systems
by incorporating the possible non-uniqueness of forward evolution.

\begin{definition}
An equilibrium $x^* \in \Omega$ of \eqref{eq:f} is
\begin{enumerate}
\item 
{\em Lyapunov stable} if for every neighbourhood $\cN_1 \subset \Omega$ of $x^*$
there exists a neighbourhood $\cN_2 \subset \cN_1$ of $x^*$
such that for every Filippov solution $\phi(t)$ that satisfies $\phi(0) \in \cN_2$,
we have $\phi(t) \in \cN_1$ for all $t \ge 0$,
\item
{\em asymptotically stable} if it is Lyapunov stable and there exists a neighbourhood
$\cN_3$ of $x^*$ such that $\phi(t) \to x^*$ as $t \to \infty$ for every Filippov solution $\phi(t)$ with $\phi(0) \in \cN_3$, and
\item
{\em unstable} if it is not Lyapunov stable.
\end{enumerate}
\label{df:stability}
\end{definition}

\subsection{Division of the discontinuity surface}

As described in \S\ref{sec:intro}, the Lie derivatives of $f^L$ and $f^R$ with respect to $H$
capture the behaviour of orbits relative to the discontinuity surface $\Sigma$.
For our purposes, Lie derivatives are defined as follows (for details and extensions see \cite{Ya20}).
Let $f$ be a smooth vector field, and let $\varphi_t(x)$ be the flow induced by $\dot{x} = f(x)$.
Also let $h$ be a smooth scalar function.
Then for any $m \ge 0$ the $m^{\rm th}$ {\em Lie derivative} of $h$ with respect to $f$ is
\begin{equation}
\cL_f^m h(x) = \frac{\partial^m}{\partial t^m} \,h(\varphi_t(x)) \Big|_{t=0}.
\label{eq:LieDeriv}
\end{equation}
In applications, Lie derivatives are usually evaluated directly from the vector field via the recurrence relation
\begin{equation}
\cL_f^m h(x) = \left( \nabla \cL_f^{m-1} h(x) \right)^{\sf T} f(x),
\label{eq:LieDeriv2}
\end{equation}
valid for all $m \ge 1$.

The first Lie derivatives, $\cL_{f^L} H(x)$ and $\cL_{f^R} H(x)$,
give the speed and direction of orbits relative to $\Sigma$.
From these we define the following regions, illustrated in Fig.~\ref{fig:B} for a two-dimensional system.

\begin{definition}
Consider a piecewise-$C^0$ system \eqref{eq:f}.
A subset $S \subset \Sigma$ is
\begin{enumerate}
\item
a {\em crossing region} if $\cL_{f^L} H(x) \cL_{f^R} H(x) > 0$ for all $x \in S$,
\item
an {\em attracting sliding region} if $\cL_{f^L} H(x) > 0$ and $\cL_{f^R} H(x) < 0$ for all $x \in S$, and
\item
a {\em repelling sliding region} if $\cL_{f^L} H(x) < 0$ and $\cL_{f^R} H(x) > 0$ for all $x \in S$.
\end{enumerate}
\end{definition}

\begin{figure}
\begin{center}
\includegraphics[width=10cm]{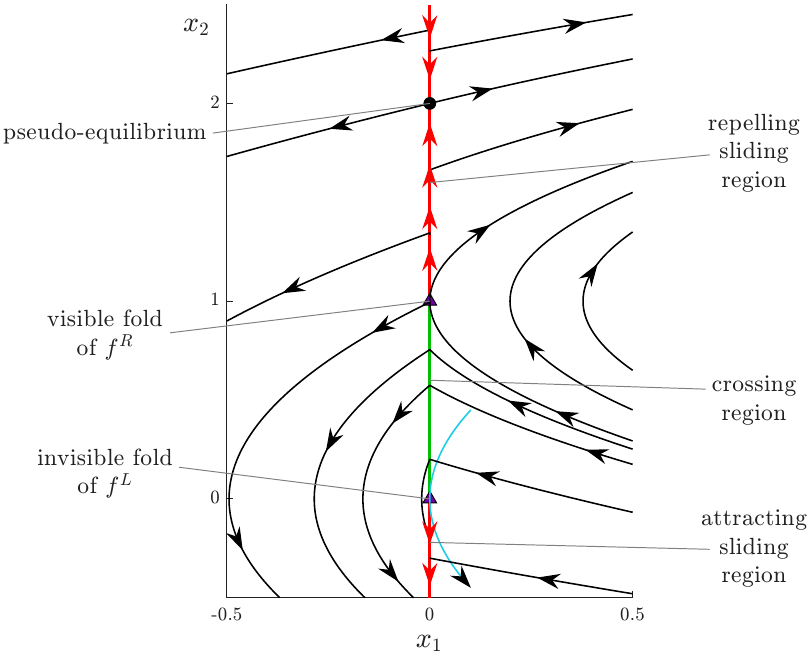}
\caption{
A phase portrait of \eqref{eq:f} with \eqref{eq:B}.
The blue trajectory (virtual) is the orbit of $\dot{x} = f^L(x)$
passing through the invisible fold of $f^L$.
\label{fig:B}
} 
\end{center}
\end{figure}

\subsection{Tangency points}

A point $x \in \Sigma$ for which $\cL_{f^L} H(x) = 0$ or $\cL_{f^R} H(x) = 0$ is a {\em tangency point}.
The collection of these form the {\em tangency surfaces}
\begin{align}
\cT_L &= \left\{ x \in \Sigma \,\middle|\, \cL_{f^L} H(x) = 0 \right\}, \\
\cT_R &= \left\{ x \in \Sigma \,\middle|\, \cL_{f^R} H(x) = 0 \right\}.
\label{eq:TLTR}
\end{align}
If
\begin{equation}
{\rm rank} \left( \begin{bmatrix} \nabla H(x) & \nabla \cL_{f^L} H(x) \end{bmatrix} \right) = 2,
\label{eq:rankTwoL}
\end{equation}
for all $x \in \cT_L$, then $\cT_L$ is a $C^k$ codimension-two manifold by the regular value theorem \cite{Hi76}.
Similarly if
\begin{equation}
{\rm rank} \left( \begin{bmatrix} \nabla H(x) & \nabla \cL_{f^R} H(x) \end{bmatrix} \right) = 2,
\label{eq:rankTwoR}
\end{equation}
for all $x \in \cT_R$, then $\cT_R$ is a $C^k$ codimension-two manifold.

To classify the local behaviour of orbits as they pass through tangency points, we look to the second Lie derivatives of $f^L$ and $f^R$.
For any $x \in \cT_L$, if $\cL_{f^L}^2 H(x) \ne 0$
then the orbit of $f^L$ through $x$ has a quadratic tangency to $\Sigma$ at $x$, and we refer to $x$ as a {\em fold}.
If $\cL_{f^L}^2 H(x) < 0$, then locally this orbit belongs to the half-space $H \le 0$,
so is a Filippov solution to \eqref{eq:f}, and we say $x$ is {\em visible}.
If instead $\cL_{f^L}^2 H(x) > 0$, the orbit locally belongs to $H \ge 0$,
so is not a Filippov solution to \eqref{eq:f}, and we say $x$ is {\em invisible}.
For $f^R$, visibility is determined from $\cL_{f^R}^2 H(x)$ but with the opposite sign.
Formally we have the following definition.

\begin{definition}
Consider a piecewise-$C^1$ system \eqref{eq:f}.
A point $x \in \cT_L$ [$\cT_R$] is 
\begin{enumerate}
\item 
a {\em visible fold} if $\cL_{f^L}^2 H(x) < 0$ [$\cL_{f^R}^2 H(x) > 0$], and
\item
an {\em invisible fold} if $\cL_{f^L}^2 H(x) > 0$ [$\cL_{f^R}^2 H(x) < 0$].
\end{enumerate}
\label{df:fold}
\end{definition}

\begin{example}
Consider \eqref{eq:f} with $x = (x_1,x_2)$ and
\begin{align}
f^L(x) &= \begin{bmatrix} -x_2 \\ -1 \end{bmatrix}, &
f^R(x) &= \begin{bmatrix} 2 x_2 - 2 \\ 1 \end{bmatrix}, &
H(x) &= x_1 \,,
\label{eq:B}
\end{align}
as in Fig.~\ref{fig:B}.
By \eqref{eq:LieDeriv2} with $m=1$,
we have $\cL_{f^L} H(x) = -x_2$ and $\cL_{f^R} H(x) = 2 x_2 - 2$.
By solving $\cL_{f^L} H(x) = 0$ and $\cL_{f^R} H(x) = 0$,
we find that $x = (0,0)$ is a tangency point of $f^L$, and $x = (0,1)$ is a tangency point of $f^R$.
By \eqref{eq:LieDeriv2} with $m=2$,
we have $\cL_{f^L}^2 H(x) = 1$ and $\cL_{f^R}^2 H(x) = 2$.
Thus $(0,0)$ is an invisible fold, and $(0,1)$ is a visible fold,
as evident in Fig.~\ref{fig:B}.
\label{exm:B}
\end{example}

\subsection{The sliding vector field}

The middle component of \eqref{eq:cF} is the set of all convex combinations
\begin{equation}
F(x,s) = (1-s) f^L(x) + s f^R(x).
\label{eq:F}
\end{equation}
While a Filippov solution is constrained to $\Sigma$,
its velocity is of the form $F(x,s)$, and is tangent to $\Sigma$.
So this velocity has $s$ such that $\cL_{F(x,s)} H(x) = 0$.
We have
\begin{equation}
\cL_{F(x,s)} H(x) = (1-s) \cL_{f^L} H(x) + s \cL_{f^R} H(x),
\nonumber
\end{equation}
so if $\cL_{f^L} H(x) \ne \cL_{f^R} H(x)$, then $F(x,s)$ is tangent to $\Sigma$ with
\begin{equation}
s = \frac{\cL_{f^L} H(x)}{\cL_{f^L} H(x) - \cL_{f^R} H(x)}.
\label{eq:s}
\end{equation}
By substituting \eqref{eq:s} into \eqref{eq:F}, we obtain
\begin{equation}
f^S(x) = \frac{f^R(x) \cL_{f^L} H(x) - f^L(x) \cL_{f^R} H(x)}{\cL_{f^L} H(x) - \cL_{f^R} H(x)},
\label{eq:fS}
\end{equation}
known as the {\em sliding vector field}.
On sliding regions, we have $s \in (0,1)$, and thus Filippov solutions $\phi(t)$ obey $\dot{\phi}(t) = f^S(\phi(t))$.

\begin{definition}
A {\em pseudo-equilibrium} of \eqref{eq:f} is a point $x^* \in \Sigma$
for which $\cL_{f^L} H(x^*) \ne \cL_{f^R} H(x^*)$ and $f^S(x^*) = \bO$.
Moreover, $x^*$ is
\begin{enumerate}
\item
{\em admissible} if $x^*$ belongs to a sliding region (attracting or repelling), and
\item
{\em virtual} if $x^*$ belongs to a crossing region.
\end{enumerate}
\label{df:admissiblePsEq}
\end{definition}

If $x^*$ is an admissible pseudo-equilibrium,
then it is an equilibrium of \eqref{eq:f} in the sense of Definition \ref{df:eq}.
If $x^*$ is a virtual pseudo-equilibrium,
then it is not equilibrium of \eqref{eq:f}
because Filippov solutions do not obey $f^S$ on crossing regions.

In applications, sliding dynamics are usually analysed by working with
the restriction of $f^S(x)$ to $\Sigma$.
The restriction is an $(n-1)$-dimensional vector field
that we write as $\tilde{f}^S(\tilde{x})$,
where $\tilde{x} \in \mathbb{R}^{n-1}$ is a coordinate system on $\Sigma$.
In many models $\Sigma$ is linear, and it is easy to identify suitable coordinates $\tilde{x}$.
If $\Sigma$ is nonlinear, such coordinates exist locally.

The following result characterises the stability of pseudo-equilibria
by applying smooth dynamical systems theory to sliding motion,
and taking into account the attracting/repelling nature of the sliding region.
We write $\tilde{x}^* \in \mathbb{R}^{n-1}$ for $x^* \in \mathbb{R}^n$ in $\tilde{x}$-coordinates.

\begin{theorem}
Let $x^* \in \Sigma$ be an admissible pseudo-equilibrium of a piecewise-$C^1$ system \eqref{eq:f}.
\begin{enumerate}
\item
If $x^*$ belongs to an attracting sliding region and all eigenvalues of $\rD \tilde{f}^S \left( \tilde{x}^* \right)$
have negative real-part, then $x^*$ is an asymptotically stable equilibrium of \eqref{eq:f}.
\item
If $x^*$ belongs to a repelling sliding region,
or $\rD \tilde{f}^S \left( \tilde{x}^* \right)$ has an eigenvalue with positive real-part,
then $x^*$ is an unstable equilibrium of \eqref{eq:f}.
\end{enumerate}
\label{th:stab}
\end{theorem}

\begin{proof}~
\begin{enumerate}
\item
By the eigenvalue assumption, $\tilde{x}^*$ is an asymptotically stable equilibrium of
the smooth system $\dot{\tilde{x}} = \tilde{f}^S(\tilde{x})$.
Thus for any neighbourhood $\cN_1 \subset \Omega$ of $x^*$,
there exists a neighbourhood $\cN_2 \subset \cN_1$ of $x^*$ such that the forward orbit $\phi(t)$ of every $x \in \cN_2 \cap \Sigma$
satisfies $\phi(t) \in \cN_1 \cap \Sigma$ for all $t \ge 0$, and $\phi(t) \to x^*$ as $t \to \infty$.
Since $x^*$ belongs to an attracting sliding region and $\cL_{f^L} H$ and $\cL_{f^R} H$ are continuous,
there exists a neighbourhood $\cN_3 \subset \cN_2$ of $x^*$
such that the forward orbit of every $x \in \cN_3$ reaches $\cN_2$ in finite time.
These orbits converge to $x^*$ without exiting $\cN_1$, thus $x^*$ is asymptotically stable.
\item
If $x^*$ belongs to a repelling sliding region,
the forward orbit of $x^*$ under $f^L$ is a Filippov solution that emanates from $x^*$, thus $x^*$ is unstable.
If $\rD \tilde{f}^S \left( \tilde{x}^* \right)$ has an eigenvalue with positive real-part,
then $\tilde{x}^*$ is unstable for $\dot{\tilde{x}} = \tilde{f}^S(\tilde{x})$,
thus $x^*$ is unstable for \eqref{eq:f}.
\end{enumerate}
\end{proof}

\begin{example}
Here we compute and analyse pseudo-equilibria of \eqref{eq:f} with \eqref{eq:B}, continuing from Example \ref{exm:B}.
By evaluating \eqref{eq:fS}, we obtain
\begin{equation}
f^S(x) = \frac{1}{2 - 3 x_2} \begin{bmatrix} 0 \\ x_2 - 2 \end{bmatrix},
\label{eq:fSexmB}
\end{equation}
which is well-defined for $x_2 \ne \frac{2}{3}$.
The vector field \eqref{eq:fSexmB} vanishes at $x^* = (0,2)$,
which belongs to a sliding region, thus this point is an admissible pseudo-equilibrium of \eqref{eq:f}.
Using $x_2$ for the coordinate $\tilde{x}$,
we have $\tilde{f}^S(x_2) = \frac{x_2 - 2}{2 - 3 x_2}$.
The derivative of $\tilde{f}^S$ at $x^*$
is $\rD \tilde{f}^S(2) = -\frac{1}{4}$.
This value is negative, so nearby sliding orbits converge to $x^*$, as indicated in Fig.~\ref{fig:B}.
However, $x^*$ is unstable because it belongs to a repelling sliding region.
\label{exm:Bcontinued}
\end{example}

\section{Refracting systems}
\label{sec:cts}

\begin{definition}
A piecewise-$C^0$ system \eqref{eq:f} on $\Omega \subset \mathbb{R}^n$ is {\em refracting} if
$\cL_{f^L} H(x) = \cL_{f^R} H(x)$ for all $x \in \Sigma$.
\label{df:secondOrder}
\end{definition}

For a refracting system, the left and right tangency surfaces \eqref{eq:TLTR} are identical.
Hence it is convenient to drop the subscripts and write
\begin{equation}
\cT = \cT_L = \cT_R \,.
\label{eq:coincidentFolds}
\end{equation}
For brevity, we write $V(x) = \cL_{f^L} H(x)$ for all $x \in \Omega$, so
\begin{equation}
V(x) = \cL_{f^L} H(x) = \cL_{f^R} H(x),
\label{eq:V}
\end{equation}
for all $x \in \Sigma$.
By \eqref{eq:LieDeriv2},
\begin{equation}
\nabla V(x)^{\sf T} f^L(x) = \cL_{f^L}^2 H(x).
\label{eq:dVL}
\end{equation}
As shown in Appendix \ref{app:dVR}, we also have
\begin{equation}
\nabla V(x)^{\sf T} f^R(x) = \cL_{f^R}^2 H(x),
\label{eq:dVR}
\end{equation}
for all $x \in \cT$ as a consequence of \eqref{eq:V}.
By \eqref{eq:rankTwoL} and \eqref{eq:rankTwoR}, the condition
\begin{equation}
{\rm rank} \left( \begin{bmatrix} \nabla H(x) & \nabla V(x) \end{bmatrix} \right) = 2,
\label{eq:rankTwoCombined}
\end{equation}
ensures that the tangency surface $\cT$ is a smooth codimension-two manifold.
With the following definition we distinguish subsets of $\cT$ by the visibility of its points.

\begin{definition}
Consider a refracting piecewise-$C^1$ system \eqref{eq:f}.
A subset $S \subset \cT$ is
\begin{enumerate}
\item 
a {\em visible-visible region} if $\cL_{f^L}^2 H(x) < 0$ and $\cL_{f^R}^2 H(x) > 0$ for all $x \in S$,
\item
a {\em visible-invisible region} if $\cL_{f^L}^2 H(x) \cL_{f^R}^2 H(x) > 0$ for all $x \in S$, and
\item
an {\em invisible-invisible region} if $\cL_{f^L}^2 H(x) >0$ and $\cL_{f^R}^2 H(x) < 0$ for all $x \in S$.
\end{enumerate}
\label{df:region}
\end{definition}

A visible-invisible region is either visible for $f^L$ and invisible for $f^R$,
or visible for $f^R$ and invisible for $f^L$.
The latter case occurs in Fig.~\ref{fig:C}.
Notice that the three types of
regions are bounded by points at which $\cL_{f^L}^2 H(x) = 0$ or $\cL_{f^R}^2 H(x) = 0$.
For a boundary point $x$ at which $\cL_{f^L}^2 H(x) = 0$ and $\cL_{f^L}^3 H(x) \ne 0$
[$\cL_{f^R}^2 H(x) = 0$ and $\cL_{f^R}^3 H(x) \ne 0$],
the orbit of $f^L$ [$f^R$] through $x$ is cubically tangent to $\Sigma$.

\begin{figure}
\begin{center}
\includegraphics[width=8cm]{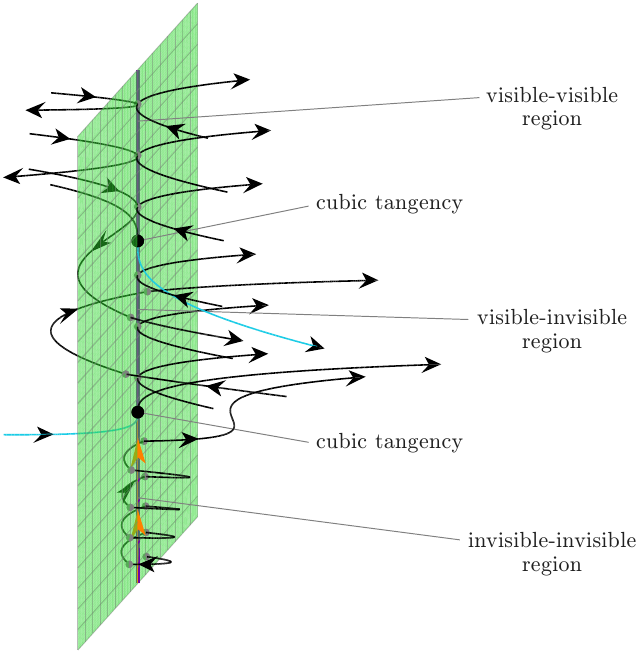}
\caption{
A phase portrait of a three-dimensional refracting system
illustrating the three types of regions defined in Definition \ref{df:region}.
This figure was drawn using \eqref{eq:C}; the coordinate axes are omitted for clarity.
The blue trajectories are the virtual parts of the cubically tangent orbits
of $\dot{x} = f^L(x)$ and $\dot{x} = f^R(x)$.
\label{fig:C}
} 
\end{center}
\end{figure}

We now construct a vector field for Filippov solutions of \eqref{eq:f} on $\cT$.
We cannot use the sliding vector field \eqref{eq:fS} because
the condition $\cL_{f^L} H(x) \ne \cL_{f^R} H(x)$ is not satisfied on $\cT$.
Instead, observe that for $x \in \cT$ any convex combination $F(x,s)$,
given by \eqref{eq:F}, is tangent to $\Sigma$.
Thus in order for $F(x,s)$ to be tangent to $\cT$, we only require $\cL_{F(x,s)} V(x) = 0$.
By \eqref{eq:F} and \eqref{eq:dVR},
\begin{equation}
\cL_{F(x,s)} V(x) = (1-s) \cL_{f^L}^2 H(x) + s \cL_{f^R}^2 H(x).
\label{eq:s2pre}
\end{equation}
So if $\cL_{f^L}^2 H(x) \ne \cL_{f^R}^2 H(x)$, then
$F(x,s)$ is tangent to $\cT$ with
\begin{equation}
s = \frac{\cL_{f^L}^2 H(x)}{\cL_{f^L}^2 H(x) - \cL_{f^R}^2 H(x)}.
\label{eq:s2}
\end{equation}
By substituting \eqref{eq:s2} into \eqref{eq:F} we obtain the {\em second-order sliding vector field}
\begin{equation}
f^T(x) = \frac{f^R(x) \cL_{f^L}^2 H(x) - f^L(x) \cL_{f^R}^2 H(x)}{\cL_{f^L}^2 H(x) - \cL_{f^R}^2 H(x)}.
\label{eq:fT}
\end{equation}

If \eqref{eq:f} is piecewise-$C^k$, then $f^T$ is $C^{k-1}$.
On visible-visible and invisible-invisible regions,
we have $\cL_{f^L}^2 H(x) \cL_{f^L}^2 H(x) < 0$,
and thus $s \in (0,1)$.
In this case $f^T(x) \in \cF(x)$,
therefore orbits may slide along visible-visible and invisible-invisible regions,
and while doing so obey $\dot{x} = f^T(x)$.
On visible-invisible regions, no member of $\cF(x)$ is tangent to $\cT$,
thus orbits cannot slide along these regions.

Finally in this section we relate $f^T$ to the sliding vector field of Carvalho {\em et al.}~\cite{CaNo24}.
Carvalho {\em et al.}~considered sliding motion of a system \eqref{eq:f}
on a smooth codimension-$m$ surface $M = \left\{ x \in \mathbb{R}^n \,\middle|\, \eta(x) = \bO \right\}$.
They smoothed the system
and showed that in the nonsmooth limit orbits on $M$ obey the {\em tangential sliding vector field}
\begin{equation}
f^{\rm tan}(x) = \frac{1-\lambda^*}{2} \,f^R(x) + \frac{1+\lambda^*}{2} \,f^L(x),
\label{eq:tsvf}
\end{equation}
where
\begin{equation}
\lambda^*(x) = \frac{\| d \eta(x) f^R(x) \| - \| d \eta(x) f^L(x) \|}{\| d \eta(x) f^R(x) \| - \| d \eta(x) f^R(x) \|}.
\label{eq:lambdaStar}
\end{equation}
We take $\eta = \left( H, V \right)$, so then $M = \cT$.

\begin{lemma}
Consider a refracting piecewise-$C^k$ system \eqref{eq:f},
and let $\eta(x) = \left( H(x), V(x) \right)$.
Then $f^{\rm tan}(x) = f^T(x)$ at any $x \in \cT$
belonging to a visible-visible or invisible-invisible region.
\label{le:CaNo24}
\end{lemma}

\begin{proof}
By \eqref{eq:dVL} and \eqref{eq:dVR}, if $x \in \cT$ then
\begin{align*}
d \eta(x) f^L(x) &= \begin{bmatrix}
\nabla H(x)^{\sf T} f^L(x) \\
\nabla V(x)^{\sf T} f^L(x)
\end{bmatrix}
= \begin{bmatrix} 0 \\ \cL_{f^L}^2 H(x) \end{bmatrix}, \\
d \eta(x) f^R(x) &= \begin{bmatrix}
\nabla H(x)^{\sf T} f^R(x) \\
\nabla V(x)^{\sf T} f^R(x)
\end{bmatrix}
= \begin{bmatrix} 0 \\ \cL_{f^R}^2 H(x) \end{bmatrix}.
\end{align*}
By substituting these into \eqref{eq:lambdaStar} we obtain
\begin{equation}
\lambda^* = \frac{\big| \cL_{f^R}^2 H(x) \big| - \big| \cL_{f^L}^2 H(x) \big|}
{\big| \cL_{f^R}^2 H(x) \big| + \big| \cL_{f^L}^2 H(x) \big|}.
\nonumber
\end{equation}
By \eqref{eq:tsvf}, $f^{\rm tan}(x) = F(x,s) = (1-s) f^L(x) + s f^R(x)$ where
\begin{equation}
s = \frac{1-\lambda^*}{2}
= \frac{\big| \cL_{f^L}^2 H(x) \big|}{\big| \cL_{f^L}^2 H(x) \big| + \big| \cL_{f^R}^2 H(x) \big|}.
\label{eq:s3}
\end{equation}
If $x$ belongs to a visible-visible region,
then $\cL_{f^L}^2 H(x) < 0$ and $\cL_{f^R}^2 H(x) > 0$, so \eqref{eq:s3} is equivalent to \eqref{eq:s2},
while if $x$ belongs to an invisible-invisible region,
then $\cL_{f^L}^2 H(x) > 0$ and $\cL_{f^R}^2 H(x) < 0$, and again \eqref{eq:s3} is equivalent to \eqref{eq:s2}.
In either case, $f^{\rm tan}(x) = F(x,s)$, where $s$ is given by \eqref{eq:s2},
so $f^{\rm tan}(x) = f^T(x)$.
\end{proof}

\section{Spiralling motion about an invisible-invisible tangency surface}
\label{sec:crossingDynamics}

In this section we analyse orbits near invisible-invisible tangency surfaces.
To do this we perform asymptotic calculations of the first return maps $P_L$ and $P_R$ illustrated in Fig.~\ref{fig:D},
and defined as follows.
Consider a refracting system \eqref{eq:f}.
Given $x \in \Sigma$ with $V(x) > 0$,
let $P_R(x)$ denote the next point in the forward orbit of $x$ under $f^R$ that belongs to $\Sigma$,
assuming this orbit ever returns to $\Sigma$,
and let $\tau_R(x)$ denote the corresponding evolution time.
Similarly, given $x \in \Sigma$ with $V(x) < 0$,
let $P_L(x)$ denote the next point in the forward orbit of $x$ under $f^L$ that belongs to $\Sigma$,
assuming this orbit ever returns to $\Sigma$,
and let $\tau_L(x)$ denote the corresponding evolution time.
Also let
\begin{equation}
P = P_L \circ P_R \,,
\label{eq:Pdefn}
\end{equation}
and 
\begin{equation}
\tau = \tau_L \circ P_R + \tau_R \,,
\label{eq:Tdefn}
\end{equation}
be the total evolution time.

\begin{theorem}
Consider a refracting piecewise-$C^3$ system \eqref{eq:f}
and let $S \subset \cT$ be a compact invisible-invisible region.
There exists a neighbourhood $\cN \subset \Omega$ of $S$
such that $P$ is well-defined for all $x \in \cN \cap \Sigma$ with $V(x) = \nu > 0$, and
\begin{align}
\tau(x) &= 2 \left( \frac{1}{\cL_{f^L}^2 H(x)} - \frac{1}{\cL_{f^R}^2 H(x)} \right) \nu + \cO \left( \nu^2 \right), \label{eq:T} \\
P(x) &= x + 2 \left( \frac{f^L(x)}{\cL_{f^L}^2 H(x)} - \frac{f^R(x)}{\cL_{f^R}^2 H(x)} \right) \nu + \cO \left( \nu^2 \right), \label{eq:P} \\
V(P(x)) &= \nu + \frac{2 \Lambda(x) \nu^2}{3} + \cO \left( \nu^3 \right), \label{eq:VP}
\end{align}
where
\begin{equation}
\Lambda = \frac{\cL_{f^L}^3 H}{\big( \cL_{f^L}^2 H \big)^2} - \frac{\cL_{f^R}^3 H}{\big( \cL_{f^R}^2 H \big)^2},
\label{eq:Lambda2}
\end{equation}
repeating \eqref{eq:Lambda}.
\label{th:P}
\end{theorem}

Theorem \ref{th:P} is proved in Appendix \ref{app:crossing} by direct matched asymptotics.
Here we consider consequences of equation \eqref{eq:VP}
(consequences of \eqref{eq:T} and \eqref{eq:P} are described in \S\ref{sec:noZeno}).
By \eqref{eq:VP}, if $\Lambda(x) < 0$, then $V(P(x)) < V(x)$.
So upon completing one revolution about the tangency surface $\cT$,
the orbit has moved to a point with a smaller value of $V$.
Repeated revolutions give $V \to 0$, so the orbit converges to $\cT$, assuming $\Lambda < 0$ is maintained.
If instead $\Lambda > 0$, the orbit spirals outward from $\cT$.
Consequently we introduce the following definition.

\begin{definition}	
Consider a refracting piecewise-$C^3$ system \eqref{eq:f}.
A subset $S \subset \cT$ of an invisible-invisible region is
\begin{enumerate}
\item 
{\em attracting} if $\Lambda(x) < 0$ for all $x \in S$, and
\item
{\em repelling} if $\Lambda(x) > 0$ for all $x \in S$.
\end{enumerate}
\label{df:attrRep}
\end{definition}

\begin{figure}
\begin{center}
\includegraphics[width=5.6cm]{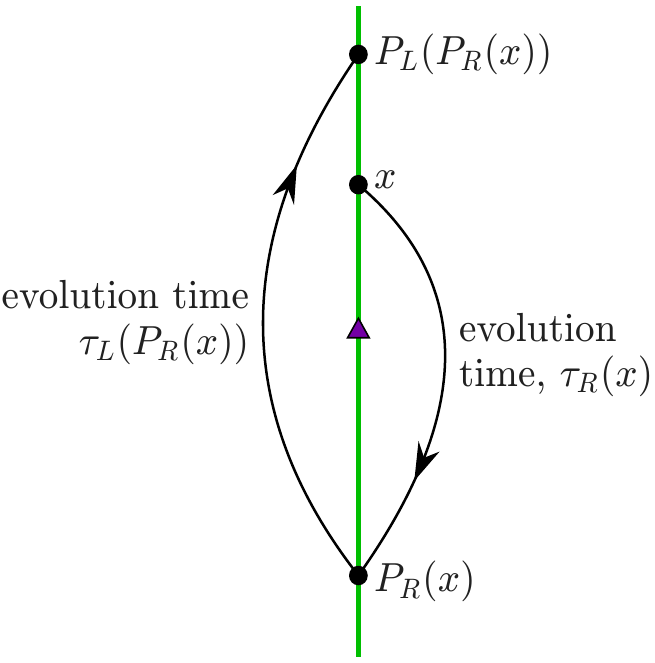}
\caption{
A sketch illustrating the first return maps $P_L$ and $P_R$.
\label{fig:D}
} 
\end{center}
\end{figure}

\begin{example}
Suppose
\begin{align}
f^L(x) &= \begin{bmatrix} x_2 \\ 1 - x_3 \\ \frac{1}{2} \end{bmatrix}, &
f^R(x) &= \begin{bmatrix} x_2 \\ x_3 \\ \frac{1}{5} \end{bmatrix}, &
H(x) &= x_1 \,,
\label{eq:C}
\end{align}
as in Fig.~\ref{fig:C}.
In this case \eqref{eq:f} is refracting and the $x_3$-axis is a tangency surface $\cT$
(coloured purple in Fig.~\ref{fig:C}).
From \eqref{eq:LieDeriv2},
\begin{align*}
\cL_{f^L} H(x) &= x_2 \,, & \cL_{f^R} H(x) &= x_2 \,, \\
\cL_{f^L}^2 H(x) &= 1-x_3 \,, & \cL_{f^R}^2 H(x) &= x_3 \,, \\
\cL_{f^L}^3 H(x) &= \tfrac{1}{2}, & \cL_{f^R}^3 H(x) &= \tfrac{1}{5}.
\end{align*}
So by Definition \ref{df:region} the $x_3$-axis is a visible-visible region for $x_3 > 1$,
a visible-invisible region for $0 < x_3 < 1$,
and an invisible-invisible region for $x_3 < 0$.

By evaluating the second-order sliding vector field \eqref{eq:fT} on the $x_3$-axis, we obtain
\begin{equation}
f^T(x) = \frac{1}{1 - 2 x_3} \begin{bmatrix} 0 \\ 0 \\ \frac{1}{5} - \frac{7}{10} x_3 \end{bmatrix},
\nonumber
\end{equation}
which is well-defined for $x_3 \ne \frac{1}{2}$.
On the invisible-invisible region, the third component of $f^T$ is positive,
so second-order sliding orbits on this region are directed upwards.
Upon reaching the cubic tangency $x = (0,0,0)$,
such orbits switch to regular motion in the right half-space $x_1 > 0$
following the cubically tangent trajectory.

Observe
\begin{equation}
\Lambda(x) = -\frac{1}{2 (1 - x_3)^2} - \frac{1}{5 x_3^2}
\nonumber
\end{equation}
is negative.
Thus orbits near the invisible-invisible region converge to this region as they spiral around it.
However, this convergence is slow compared to the broad upwards motion.
For this reason, the spiralling orbit shown toward the bottom of
Fig.~\ref{fig:C} is hardly seen to move closer to the invisible-invisible region.
After four revolutions the orbit enters the right half-space and fails to return to $\Sigma$.
\end{example}

\section{Convergence in infinite time and consistency of the flow}
\label{sec:noZeno}

In this section we show that second-order sliding orbits
are continuous perturbations of spiralling orbits
about invisible-invisible subsets of the tangency surface $\cT$ (Theorem \ref{th:consistency}).
We also show that spiralling orbits cannot converge to $\cT$ in finite time (Theorem \ref{th:noZeno}).
This type of `no Zeno' result has been described previously for discontinuous, piecewise-linear ODEs \cite{ThCa14},
and continuous, piecewise-analytic ODEs \cite{Su82}.

We first use \eqref{eq:T} and \eqref{eq:P}
to show with a quick calculation that second-order sliding motion and spiralling motion
agree to first order over one revolution.
Let
\begin{equation}
\beta(x) = 2 \left( \frac{1}{\cL_{f^L}^2 H(x)} - \frac{1}{\cL_{f^R}^2 H(x)} \right)
\label{eq:beta}
\end{equation}
be the coefficient of the $\nu$-term in \eqref{eq:T}.
Now observe
\begin{align}
\beta(x) f^T(x) &= 2 \left( \frac{1}{\cL_{f^L}^2 H(x)} - \frac{1}{\cL_{f^R}^2 H(x)} \right)
\left( \frac{f^R(x) \cL_{f^L}^2 H(x) - f^L(x) \cL_{f^R}^2 H(x)}{\cL_{f^L}^2 H(x) - \cL_{f^R}^2 H(x)} \right) \nonumber \\
&= 2 \left( \frac{f^L(x)}{\cL_{f^L}^2 H(x)} - \frac{f^R(x)}{\cL_{f^R}^2 H(x)} \right),
\label{eq:betafT}
\end{align}
is the coefficient of the $\nu$-term in \eqref{eq:P}.
Thus mapping under $P$ is identical, at first order, to evolving under $f^T$ by the time $\tau$.

By applying this approximation several times,
we can conclude that several revolutions around $\cT$
is identical, at first order, to evolving along $\cT$ by a suitable length of time.
The following result generalises this observation to
accommodate orbits arbitrarily close to $\cT$ for which the number of revolutions over a fixed time is unbounded.
Specifically, we obtain the linear bound \eqref{eq:bound},
where $\phi(t)$ is the spiralling orbit and $\xi(t)$ is the second-order sliding orbit, see Fig.~\ref{fig:E}.
Filippov \cite[Thm.~2, \S 8]{Fi88} considers a more general setting and shows $\phi(t) \to \xi(t)$ as $x \to y$.

\begin{theorem}
Consider a refracting piecewise-$C^3$ system \eqref{eq:f}
and let $S \subset \cT$ be a compact invisible-invisible region.
Let $b > 0$ and suppose $S_0 \subset S$ has the property that
the forward orbit of every point in $S_0$ under $f^T$ remains in $S$ for all $0 \le t \le b$.
Then there exist $\delta > 0$ and $M > 0$ such that for all
$y \in S_0$ and $x \in \Omega \setminus \cT$ with $\| x - y \| < \delta$,
\begin{equation}
\| \phi(t) - \xi(t) \| \le M \| x - y \|,
\label{eq:bound}
\end{equation}
for all $0 \le t \le b$,
where $\phi(t)$ is the forward orbit of $x$ under \eqref{eq:f},
and $\xi(t)$ is the forward orbit of $y$ under $f^T$.
\label{th:consistency}
\end{theorem}

Theorem \ref{th:consistency} is proved in Appendix \ref{app:consistency}
by showing that the errors produced by repeated applications
of the above first-order approximation do not accumulate unfavourably.
Conceptually, the proof is similar that commonly used to bound
the global truncation error of the forward Euler numerical method \cite{At89}.

The next result shows that infinitely many switches cannot occur in finite time.

\begin{theorem}
Consider a refracting piecewise-$C^3$ system \eqref{eq:f} on $\Omega \subset \mathbb{R}^n$.
Suppose the forward [backward] orbit of a point $x \in \Omega \setminus \cT$ intersects $\cT$ in finite time.
Then its first point of intersection with $\cT$ does not belong to an invisible-invisible region.
\label{th:noZeno}
\end{theorem}

\begin{figure}
\begin{center}
\includegraphics[width=12cm]{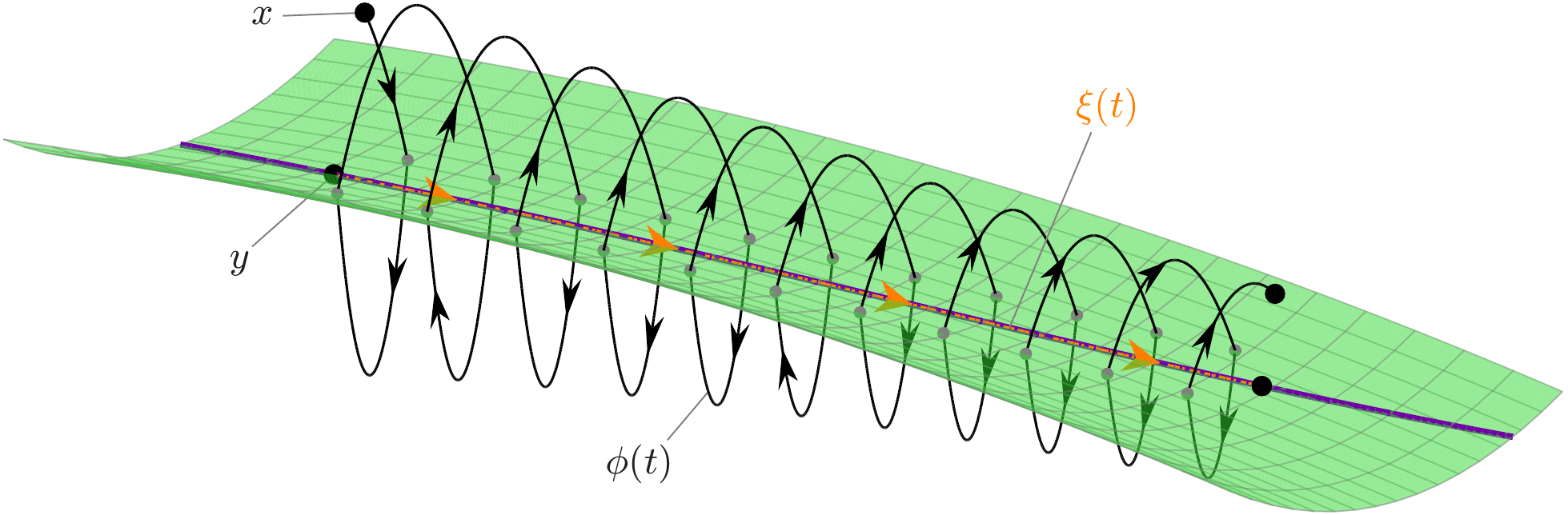}
\caption{
A sketch illustrating the orbits $\phi(t)$ and $\xi(t)$ of Theorem \ref{th:consistency}.
\label{fig:E}
} 
\end{center}
\end{figure}

Before we prove Theorem \ref{th:noZeno}, we provide the following heuristic argument.
Consider the map $V_{j+1} = V_j - a V_j^2$, where $a > 0$ is fixed, as a one-dimensional approximation to \eqref{eq:VP}
in the case of an attracting invisible-invisible region.
It is a simple exercise in mathematical induction to show that
if the initial value obeys $0 < V_0 \le \frac{1}{2 a}$, then 
$V_j \ge \frac{V_0}{1 + 2 a V_0 j}$, for all $j \ge 0$.
With a linear approximation to \eqref{eq:T},
the total evolution time is proportional to the sum of the $V_j$,
so is bounded from below by a harmonic series.
Since harmonic series diverge, this suggests that the total evolution time is infinite.

\begin{proof}[Theorem \ref{th:noZeno}]
Without loss of generality, consider the {\em forward} orbit of $x$, call it $\phi(t)$.
Let $y = \phi(b) \in \cT$ be its first point of intersection with $\cT$,
where $b > 0$ is the (finite) intersection time.
Suppose for a contradiction $y$ belongs to an invisible-invisible region.
Then since the tangent trajectories of $f^L$ and $f^R$ at $y$ are both virtual,
the orbit reaches $y$ via an infinite sequence of switches under the map $P$.

Let $0 < t_0 < b$ be such that $x^{(0)} = \phi(t_0) \in \Sigma$ with $V \left( x^{(0)} \right) > 0$.
For each $j \ge 1$, let $x^{(j)} = P \left( x^{(j-1)} \right)$.
Notice $V \left( x^{(j)} \right) > 0$ for all $j \ge 0$, and $x^{(j)} \to y$ as $j \to \infty$.
Moreover,
\begin{equation}
t_0 + \sum_{j=0}^{N-1} \tau \left( x^{(j)} \right) \to b, ~\text{as $N \to \infty$}.
\label{eq:noZenoProof10}
\end{equation}

By \eqref{eq:VP}, there exist $a_1, a_2 > 0$ and a neighbourhood $\cN \subset \mathbb{R}^n$ of $y$ such that
for all $z \in \cN \cap \Sigma$ with $V(z) > 0$,
\begin{align}
V(P(z)) &\ge V(z) - a_1 V(z)^2, \label{eq:noZenoProof20} \\
\tau(z) &\ge a_2 V(z). \label{eq:noZenoProof21}
\end{align}
Since $x^{(j)} \to y$, there exists $m \ge 0$
such that $x^{(j)} \in \cN$ and $V \left( x^{(j)} \right) < \frac{1}{2 a_1}$ for all $j \ge m$.
Let $p \ge 1$ be such that $V \left( x^{(m)} \right) \ge \frac{1}{2 a_1 p}$.

We now show by induction that
\begin{equation}
V \left( x^{(m+i)} \right) \ge \frac{1}{2 a_1 (p+i)},
\label{eq:noZenoProof30}
\end{equation}
for all $i \ge 0$.
Certainly \eqref{eq:noZenoProof30} is true with $i = 0$,
and if \eqref{eq:noZenoProof30} is true for some $i \ge 0$ then, by \eqref{eq:noZenoProof20},
\begin{align}
V \left( x^{(m+i+1)} \right) &\ge V \left( x^{(m+i)} \right) - a_1 V \left( x^{(m+i)} \right)^2 \\
&\ge \frac{1}{2 a_1 (p+i)} - a_1 \left( \frac{1}{2 a_1 (p+i)} \right)^2 \\
&= \frac{p+i-1}{4 a_1 (p+i)^2 (p+i+1)} + \frac{1}{2 a_1 (p+i+1)} \\
&\ge \frac{1}{2 a_1 (p+i+1)},
\end{align}
so \eqref{eq:noZenoProof30} is also true for $i+1$.
Thus by induction \eqref{eq:noZenoProof30} is true for all $i \ge 0$.
Then by \eqref{eq:noZenoProof21}, $\tau \left( x^{(m+i)} \right) \ge \frac{a_2}{2 a_1 (p+i)}$ for all $i \ge 0$.
Thus
\begin{equation}
\sum_{i=0}^{N-1} \tau \left( x^{(m+i)} \right) \ge \frac{a_2}{2 a_1} \sum_{i=0}^{N-1} \frac{1}{p+i}
\nonumber
\end{equation}
tends to infinity as $N \to \infty$, contradicting \eqref{eq:noZenoProof10}.
\end{proof}

\section{Second-order pseudo-equilibria}
\label{sec:pseq}

A (first-order) pseudo-equilibrium is a zero of the sliding vector field $f^S$ (Definition \ref{df:admissiblePsEq}).
If it belongs to a sliding region, then it is an equilibrium of \eqref{eq:f}
in the sense of being a Filippov solution to \eqref{eq:f}, and we say it is admissible.
Here we generalise this to the second-order sliding vector field $f^T$.

\begin{definition}
A {\em second-order pseudo-equilibrium} of a refracting system \eqref{eq:f} is a point $x^* \in \cT$
for which $\cL_{f^L}^2 H(x^*) \ne \cL_{f^R}^2 H(x^*)$ and $f^T(x^*) = \bO$.
Moreover, $x^*$ is
\begin{enumerate}
\item
{\em admissible} if $x^*$ belongs to a visible-visible or invisible-invisible region, and
\item
{\em virtual} if $x^*$ belongs to a visible-invisible region.
\end{enumerate}
\label{df:admissibleSOPsEq}
\end{definition}

If $x^*$ is admissible, then it is an equilibrium of \eqref{eq:f} because $f^T(x^*) \in \cF(x^*)$,
while if $x^*$ is virtual, then it is not equilibrium of \eqref{eq:f}.

The following result characterises the stability of second-order pseudo-equilibria, analogous to Theorem \ref{th:stab} for pseudo-equilibria.
The restriction of $f^T$ to $\cT$ is an $(n-2)$-dimensional vector field
that we write as $\tilde{f}^T(\tilde{x})$, where $\tilde{x}$ is $(n-2)$-dimensional,
and we let $\tilde{x}^*$ denote $x^*$ in $\tilde{x}$-coordinates.

\begin{theorem}
Let $x^* \in \cT$ be a second-order pseudo-equilibrium of a refracting piecewise-$C^3$ system \eqref{eq:f}.
\begin{enumerate}
\item
If $x^*$ belongs to an attracting invisible-invisible region and
all eigenvalues of $\rD \tilde{f}^T \left( \tilde{x}^* \right)$ have negative real-part,
then $x^*$ is an asymptotically stable equilibrium of \eqref{eq:f}.
\item
If $x^*$ belongs to a visible-visible region,
a repelling invisible-invisible region,
or $\rD \tilde{f}^T \left( \tilde{x}^* \right)$ has an eigenvalue with positive real-part,
then $x^*$ unstable.
\end{enumerate}
\label{th:stab2}
\end{theorem}

As a visual aid, Fig.~\ref{fig:F} shows orbits near a second-order pseudo-equilibrium $x^*$
of a three-dimensional system.
Here $\rD \tilde{f}^T \left( \tilde{x}^* \right)$ is scalar and negative,
so second-order sliding orbits converge to $x^*$.
The equilibrium $x^*$ belongs to a repelling invisible-invisible region
so is unstable and the nearby sprialling orbit diverges.

To prove Theorem \ref{th:stab2}, we first show that in case (1) $x^*$ is Lyapunov stable. 
This is done in Appendix \ref{app:LyapStable} by using Theorem \ref{th:consistency} to show that spiralling orbits
are well approximated by second-order sliding orbits,
that we know converge exponentially to $x^*$ by the eigenvalue assumption.

\begin{proof}[Theorem \ref{th:stab2}]~
\begin{enumerate}
\item
By Lemma \ref{le:noEscape}, $x^*$ is Lyapunov stable.
Thus for any neighbourhood $\cN_1 \subset \Omega$ of $x^*$
there exists a neighbourhood $\cN_2 \subset \cN_1$ of $x^*$ such that the forward orbit of every $x \in \cN_2$ does not escape $\cN_1$.
In view of the eigenvalue condition we can assume $\cN_2$ is small enough
that the forward orbit of every $x \in \cN_2 \cap \cT$ under $f^T$ converges to $x^*$ as $t \to \infty$.
It remains to show that the forward orbit of every $x \in \cN_2 \setminus \cT$ converges to $x^*$ as $t \to \infty$.

Since $\Lambda(x^*) < 0$ we can assume $\cN_2$ is small enough
that $V(P(x)) \le V(x) - a V(x)^2$ for all $x \in \cN_2 \cap \Sigma$ with $V(x) > 0$, for some constant $a > 0$.
Since iterations of $V_{j+1} = V_j - a V_j^2$ converge to $V = 0$ for any sufficiently small $V_0 > 0$,
the forward orbit of any $x \in \cN_2 \cap \Sigma$ under \eqref{eq:f} converges to $\cN_2 \cap \cT$ as $t \to \infty$.
Thus $\omega(x)$ (the $\omega$-limit set of $x$) is non-empty and contained in $\cN_2 \cap \cT$.

We now show that $\omega(x)$ is forward invariant under $f^T$.
Choose any $z \in \omega(x)$.
Let $\psi(t)$ be the forward orbit of $z$ under $f^T$,
and let $\phi(t)$ be the forward orbit of $x$ under \eqref{eq:f}.
Let $t_k$ be an increasing and unbounded sequence of times
with the property that $\phi(t_k) \to z$ as $k \to \infty$,
and choose any $t > 0$.
Then $\left\| \phi(t_k + t) - \psi(t) \right\| \to 0$ as $k \to \infty$ by Theorem \ref{th:consistency},
thus $\psi(t) \in \omega(x)$.
Thus $\omega(x)$ is forward invariant under $f^T$,
but the only such set in $\cN_2 \cap \cT$ is $\{ x^* \}$.
Thus $\omega(x) = \{ x^* \}$, that is, the forward orbit of $x$ under \eqref{eq:f} converges to $x^*$.
Thus $x^*$ is asymptotically stable.
\item
If $x^*$ belongs to a repelling invisible-invisible region,
then there exists $a > 0$ and a neighbourhood $\cN_1 \subset \Omega$ of $x^*$ such that
$V(P(x)) \ge V(x) + a V(x)^2$ for $x \in \cN_1 \cap \Sigma$ with $V(x) > 0$.
Since iterations of $V_{j+1} = V_j + a V_j^2$ escape a neighbourhood of $V = 0$ for any sufficiently small $V_0 > 0$,
there exists a neighbourhood $\cN_2 \subset \cN_1$ of $x^*$ such that
the forward orbits of all $x \in \cN_2 \cap \Sigma$ with $V(x) > 0$ escape $\cN_2$, and so $x^*$ is unstable.
If $x^*$ belongs to a visible-visible region
then $x^*$ is unstable because the orbit of $f^R$ emanating from $x^*$
is an (admissible) solution to \eqref{eq:f}.
Finally if $\rD \tilde{f}^T \left( \tilde{x}^* \right)$ has an eigenvalue with positive real-part,
then $\tilde{x}^*$ is unstable for $\tilde{f}^T$,
and thus $x^*$ is unstable for \eqref{eq:f}.
\end{enumerate}
\end{proof}

\begin{figure}
\begin{center}
\includegraphics[width=10.8cm]{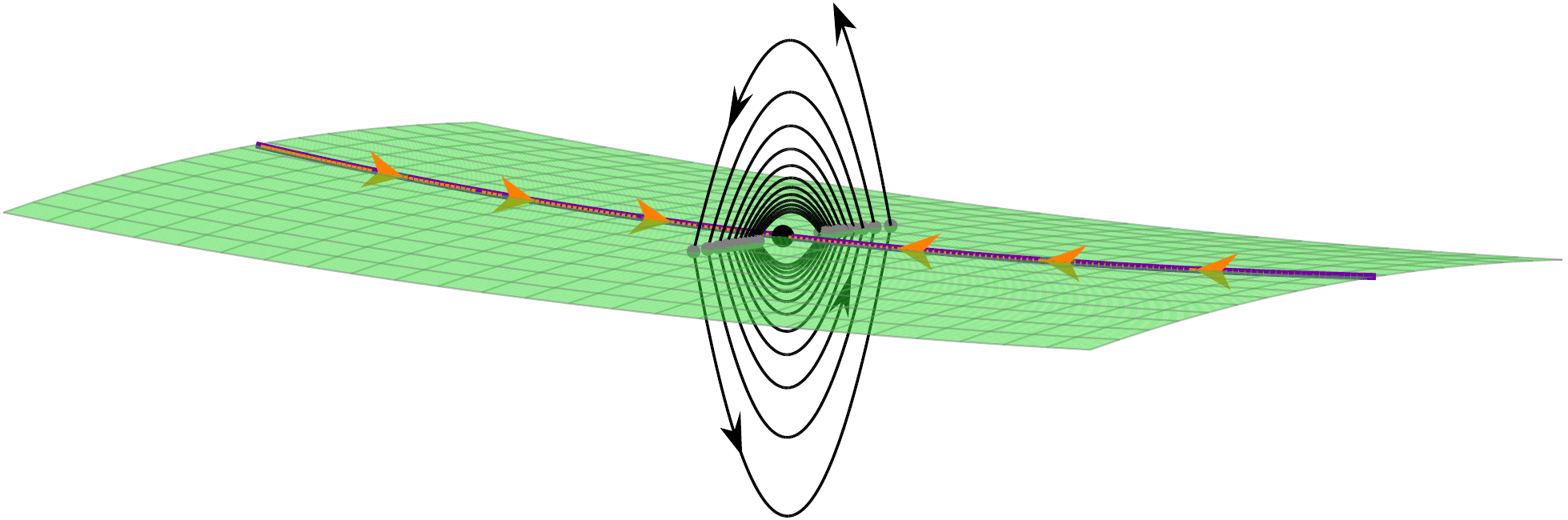}
\caption{
A phase portrait of a three-dimensional refracting system
near a second-order pseudo-equilibrium belonging to a repelling invisible-invisible region.
\label{fig:F}
} 
\end{center}
\end{figure}

\section{Application to mechanical systems with impacts}
\label{sec:impacts}

Many mechanical systems contain rigid components that bump into each other
causing complicated dynamics including chaos.
To understand this behaviour it can be helpful to analyse low-dimensional models.
The models often combine ODEs for free motion, with either a reset law (map) for impact events leading to a hybrid model,
or a different set of ODEs for in-contact motion leading to a piecewise-smooth model \cite{DiBu08}.

In the latter case the model is often a refracting Filippov system.
This is because if the discontinuity surface of a system \eqref{eq:f} represents a certain location of an object,
then the Lie derivatives $\cL_{f^L} H(x)$ and $\cL_{f^R} H(x)$ represent the velocity of the object
relative to the location, so are identical for $f^L$ and $f^R$.

As an example we consider the non-dimensionalised model
\begin{equation}
\ddot{x}_1 + b \dot{x}_1 + k (x_1 + 1) = \cA \cos(t) +
\begin{cases}
0, & x_1 < 0, \\
-k_D (x_1 + d), & x_1 > 0,
\end{cases}
\label{eq:impactOsc}
\end{equation}
of the block-damper system shown in Fig.~\ref{fig:BlockDamper}.
The position $x_1(t)$ of the block is treated as a harmonically forced linear oscillator with equilibrium position $x_1 = -1$.
When $x_1(t) > 0$ the block is assumed to be in contact with a massless damper that is prestressed by a distance $d > 0$.

If the forcing amplitude $\cA$ is less than $\cA_{\rm graz} = \sqrt{(k-1)^2 + b^2}$,
the system has a stable non-impacting periodic orbit.
At $\cA = \cA_{\rm graz}$ this orbit undergoes a grazing bifurcation
leading to complicated long-term dynamics \cite{MaAg06,MaIn08,MoDe01,PaIn10}.
Here we consider yet larger values of $\cA$
with which the acceleration of the block at $x_1 = 0$ can be positive and second-order sliding motion can occur.

We first write the model in the form \eqref{eq:f}.
Let $x = (x_1,x_2,x_3)$, where $x_2 = \dot{x}_1$ is the velocity of the block,
and $x_3 = t \,{\rm mod}\, 2 \pi$ is the phase of the forcing.
Then \eqref{eq:impactOsc} is equivalent to \eqref{eq:f} with
\begin{align}
f^L(x) &= \begin{bmatrix}
x_2 \\
-k (x_1 + 1) - b x_2 + \cA \cos(x_3) \\
1 \end{bmatrix}, &
g(x) &= \begin{bmatrix}
0 \\
-k_D (x_1 + d) \\
0 \end{bmatrix}, &
H(x) &= x_1 \,,
\label{eq:impactOscfLgH}
\end{align}
and $f^R(x) = f^L(x) + g(x)$.
The discontinuity surface $\Sigma$ consists of all points $x$ for which $x_1 = 0$.
Observe $\cL_{f^L} H(x) = \cL_{f^R} H(x) = x_2$, so the system is refracting,
and the tangency surface $\cT$ consists of all points $x$ for which $x_1 = x_2 = 0$.
On $\cT$ the second Lie derivatives are
\begin{align}
\cL_{f^L}^2 H(x) \big|_{\cT} &= \cA \cos(x_3) - k , &
\cL_{f^R}^2 H(x) \big|_{\cT} &= \cA \cos(x_3) - k - k_D d.
\end{align}
Thus subsets of $\cT$ for which $k < \cA \cos(x_3) < k + k_D d$ are invisible-invisible regions.
Such a region is indicated in Fig.~\ref{fig:Imp}:
it is bounded by points $c^L$ and $c^R$ where $\cL_{f^L}^2 H(x) = 0$ and $\cL_{f^R}^2 H(x) = 0$ respectively.
On this region Filippov solutions obey the second-order sliding vector field $f^T$,
and by evaluating the formula \eqref{eq:fT} on $\cT$ we obtain simply
\begin{equation}
f^T(x) = \begin{bmatrix} 0 \\ 0 \\ 1 \end{bmatrix}.
\label{eq:impactOscfT}
\end{equation}

\begin{figure}
\begin{center}
\includegraphics[width=10.8cm]{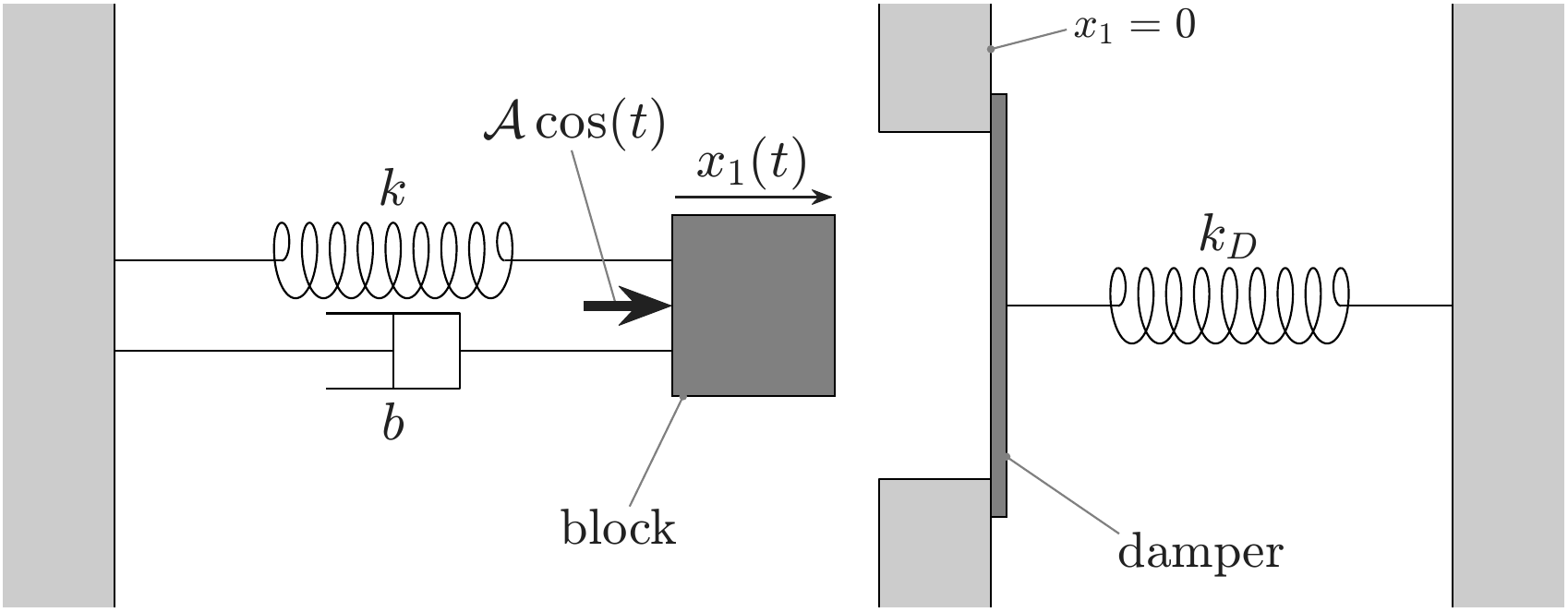}
\caption{
The block-damper system modelled by \eqref{eq:impactOsc}.
\label{fig:BlockDamper}
} 
\end{center}
\end{figure}

On $\cT$ the third Lie derivatives are
$\cL_{f^L}^3 H(x) \big|_{\cT} = \cL_{f^R}^3 H(x) \big|_{\cT} = -\cA \sin(x_3)$.
Thus the value \eqref{eq:Lambda} is
\begin{equation}
\Lambda = -\cA \sin(x_3) \left( \frac{1}{\left( \cA \cos(x_3) - k \right)^2}
- \frac{1}{\left( \cA \cos(x_3) - k - k_D d \right)^2} \right).
\label{eq:impactOscLambda}
\end{equation}
Above the point $\chi$ indicated in Fig.~\ref{fig:Imp},
we have $\Lambda < 0$, and so the invisible-invisible region is attracting.
Below $\chi$, $\Lambda > 0$, and the region is repelling.

Fig.~\ref{fig:Imp} also shows part of a typical orbit.
The orbit spirals around the invisible-invisible region
during which time it bows outwards from $\cT$ in the $x_2$-direction while $\Lambda < 0$,
and bends back toward $\cT$ while $\Lambda > 0$.
Eventually the orbit leaves the vicinity of the invisible-invisible region and escapes into the left half-space $x_1 < 0$.

\begin{figure}
\begin{center}
\includegraphics[width=12cm]{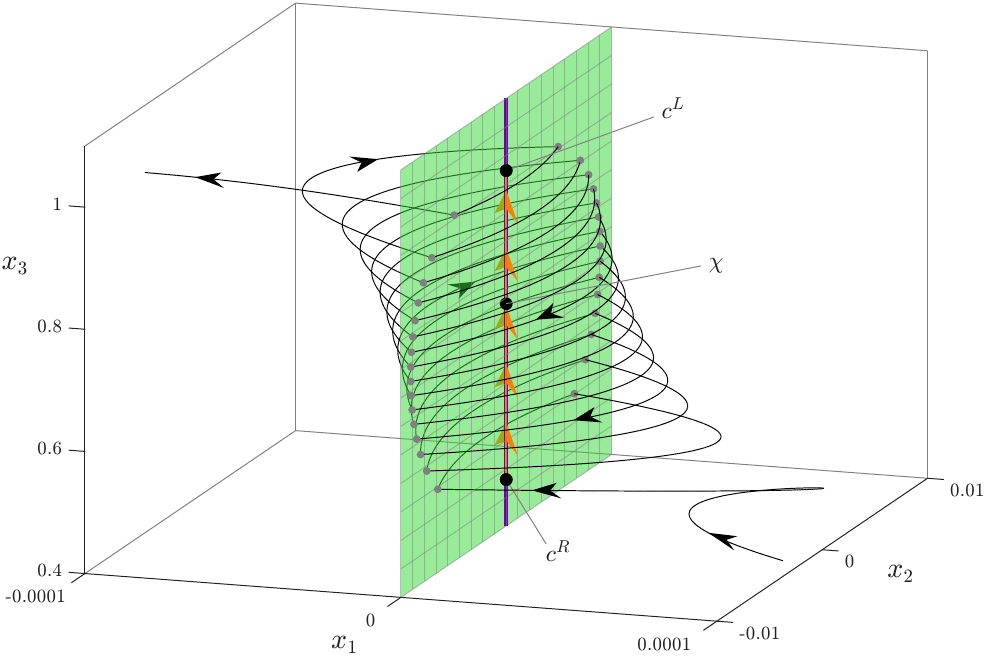}
\caption{
An orbit of the impact oscillator model \eqref{eq:impactOsc}
with $x_2 = \dot{x}_1$ and $x_3 = t \,{\rm mod}\, 2 \pi$,
and parameters $k = 5$, $b = 0.5$, $k_D = 10$, $d = 0.3$, and $\cA = 9$.
Between the points $c^L$ and $c^R$ (where $f^L$ and $f^R$, respectively, are cubically tangent to the discontinuity surface),
the tangency surface $\cT$ is invisible-invisible.
Between $c^L$ and $\chi$, $\cT$ is attracting;
between $c^R$ and $\chi$, $\cT$ is repelling.
\label{fig:Imp}
} 
\end{center}
\end{figure}

\section{Application to ant colony migrations}
\label{sec:ants}

Consider \eqref{eq:f} with $x = (x_1,x_2,x_3)$ and
\begin{equation}
\begin{split}
f^L(x) &= \begin{bmatrix}
\left( \alpha_{s a} + \beta_{\ell s} x_2 + \beta_{c s} x_3 \right) (\rho N - x_1 - x_2 - x_3) - (\alpha_{a s} + \alpha_{a \ell}) x_1 \\
\alpha_{a \ell} x_1 - \alpha_{\ell s} x_2 \\
-\alpha_{c s} x_3
\end{bmatrix}, \\
g(x) &= \begin{bmatrix}
0 \\ -\alpha_{\ell c} x_2 \\ \alpha_{\ell c} x_2
\end{bmatrix}, \qquad \qquad
H(x) = x_1 + x_2 + x_3 - \Theta,
\end{split}
\label{eq:antfLgH}
\end{equation}
and $f^R(x) = f^L(x) + g(x)$,
where $N, \Theta, \rho, \alpha_{c s}, \alpha_{\ell s}, \beta_{\ell s}, \beta_{c s}, \alpha_{a \ell}, \alpha_{\ell c}, \alpha_{a s}, \alpha_{s a} > 0$ are parameters.
This is a simplified ant colony migration model of Wang {\em et al.}~\cite{WaQi24}.
It is a compartmental model where worker ants moving around a certain site are each given a classification based on their behaviour.
In the model, $x_1$ is the number of {\em assessing} ants,
$x_2$ is the number of {\em leading} ants,
and $x_3$ is the number of {\em carrying} ants
(denoted $A$, $L$, and $C$ respectively in \cite{WaQi24}).
The model assumes the ants change classification at certain rates,
leading to the various terms in $f^L(x)$.
Further, the model assumes that when the number of ants at the site exceeds a threshold $\Theta$,
the colony decides collectively to relocate their home to this site.
This leads to an immediate increase in ants carrying nestmates and items from their present home to the site,
and this assumption is incorporated through the functions $g(x)$ and $H(x)$.

Wang {\em et al.}~prove that each piece of the model has a unique stable equilibrium.
The equilibrium $x^L$ of $f^L$ corresponds to the colony remaining at their present home,
while the equilibrium $x^R$ of $f^R$ corresponds to the colony having completed migration from their home to the site.
For some values of the parameters, both equilibria are virtual.
In this case, orbits of the system settle to a crossing limit cycle that corresponds to
the colony repeatedly changing their mind about whether or not migrate, Fig.~\ref{fig:Ant}.

\begin{figure}
\begin{center}
\includegraphics[width=12cm]{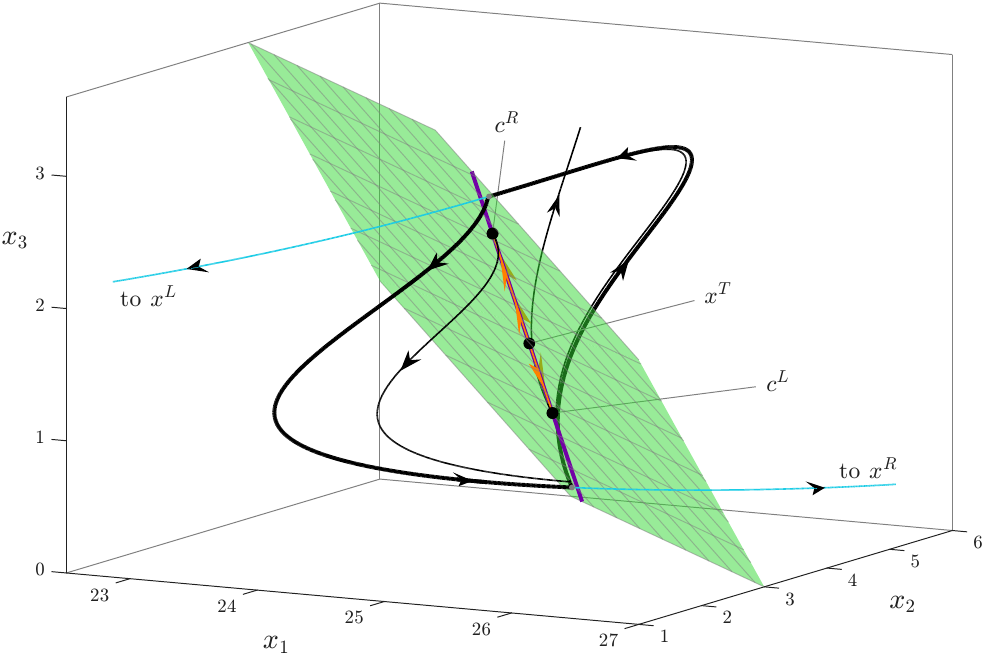}
\caption{
A phase portrait of the ant colony migration model \eqref{eq:f} with \eqref{eq:antfLgH}.
The parameter values are as in Figure 2 of \cite{WaQi24}:
$N = 200$,
$\Theta = 30$,
$\rho = 0.25$,
$\alpha_{c s} = 0.07$,
$\alpha_{\ell s} = 0.018$,
$\beta_{\ell s} = 0.049$,
$\beta_{c s} = 0.079$,
$\alpha_{a \ell} = 0.007$,
$\alpha_{\ell c} = 0.15$,
$\alpha_{a s} = 0.24$, and
$\alpha_{s a} = 0.01$.
The discontinuity surface $\Sigma$ is shaded green,
and the tangency surface $\cT$ is coloured purple.
Second-order sliding motion is coloured orange on the visible-visible part of $\cT$
and emanates from the pseudo-equilibrium $x^T$.
We show the forward orbits of the points $c^L$ and $c^R$ where $f^L$ and $f^R$ are cubically tangent to $\Sigma$.
The latter orbit converges to a stable crossing limit cycle (thick black loop).
We also show virtual extensions (blue) from the limit cycle.
These extensions converge to virtual equilibria located outside the plotted range.
\label{fig:Ant}
} 
\end{center}
\end{figure}

Here we treat the model as a Filippov system with discontinuity surface
\begin{equation}
\Sigma = \left\{ x \in \mathbb{R}^3 \,\middle|\, x_1 + x_2 + x_3 = \Theta \right\}.
\nonumber
\end{equation}
For all $x \in \Sigma$, $\nabla H(x)^{\sf T} g(x) = 0$, thus
$\cL_{f^L} H(x) = \cL_{f^R} H(x)$, so the system is refracting.
For the parameter values of Fig.~\ref{fig:Ant}, the tangency surface $\cT$ contains a visible-visible region.
By numerically evaluating the second-order sliding vector field $f^T$ on this region,
we find that second-order sliding orbits head away from a unique pseudo-equilibrium $x^T$.
This second-order sliding motion is not important physically, because here $\cT$ is visible-visible,
but it allows us to complete the phase portrait:
the stable limit cycle revolves around the unstable equilibrium $x^T$,
that corresponds to the colony forever hesitating on their choice as to whether or not to migrate.

\section{Two-dimensional systems}
\label{sec:2d}

In this section we consider the two-piece form \eqref{eq:f} in two dimensions
and compare the quantity $\Lambda$, derived above for refracting systems,
to a quantity derived in \cite{CoGa01,Si22b} that dictates
the stability of invisible-invisible planar two-folds of general Filippov systems.

Write $x = (x_1,x_2)$, and for simplicity suppose $H(x) = x_1$.
We Taylor expand $f^L$ and $f^R$ about the origin keeping all terms to second order:
for each $Z = L,R$ we write
\begin{equation}
f^Z(x) = \begin{bmatrix}
a_{0Z} + a_{1Z} x_1 + a_{2Z} x_2 + a_{3Z} x_1^2 + a_{4Z} x_1 x_2 + a_{5Z} x_2^2 \\[1.2mm]
b_{0Z} + b_{1Z} x_1 + b_{2Z} x_2 + b_{3Z} x_1^2 + b_{4Z} x_1 x_2 + b_{5Z} x_2^2
\end{bmatrix} + \cO \left( \| x \|^3 \right),
\label{eq:2dfZ}
\end{equation}
where $a_{0Z}, \ldots, b_{5Z} \in \mathbb{R}$ are constants.
If
\begin{equation}
\begin{aligned}
a_{0L} &= 0, \qquad \qquad & a_{0R} &= 0, \\
a_{2L} &> 0, \qquad \qquad & a_{2R} &> 0, \\
b_{0L} &> 0, \qquad \qquad & b_{0R} &< 0,
\end{aligned}
\label{eq:2dass1}
\end{equation}
then the origin is an invisible fold for both $f^L$ and $f^R$,
and we can consider the Poincar\'e map $P = P_L \circ P_R$ illustrated in Fig.~\ref{fig:D}.
By \cite[equation (D.7)]{Si22b},
\begin{equation}
P(r) = r + \frac{2 (\sigma_L - \sigma_R)}{3} \,r^2 + \cO \left( r^3 \right),
\label{eq:2dP}
\end{equation}
where
\begin{equation}
\sigma_Z = \frac{a_{1Z}}{b_{0Z}} + \frac{b_{2Z}}{b_{0Z}} - \frac{a_{5Z}}{a_{2Z}},
\label{eq:2dsigmaZ}
\end{equation}
for each $Z = L,R$.
Given sufficiently small $r > 0$, the forward orbit of $x = (0,r)$ under \eqref{eq:f} next returns to the positive $x_2$-axis at $x = (0,P(r))$.
Thus if $\sigma_L < \sigma_R$, then iterates of $P$ tend to $0$
and the origin is an asymptotically stable equilibrium of \eqref{eq:f},
while if $\sigma_L > \sigma_R$, then iterates of $P$ diverge from $0$ and the origin is unstable.

We now compare the sign of $\sigma_L - \sigma_R$ to the sign of $\Lambda$.
The first three Lie derivatives of $f^Z$ with respect to $H$ are
\begin{align}
\cL_{f^Z} H(x) &= \nabla H(x)^{\sf T} f^Z(x) \nonumber \\
&= a_{1Z} x_1 + a_{2Z} x_2 + a_{3Z} x_1^2 + a_{4Z} x_1 x_2 + a_{5Z} x_2^2 + \cO \left( \| x \|^3 \right), \nonumber \\
\cL_{f^Z}^2 H(x) &= \nabla \left( \cL_{f^Z} H(x) \right)^{\sf T} f^Z(x) \nonumber \\
&= \begin{bmatrix} a_{1Z} + 2 a_{3Z} x_1 + a_{4Z} x_2 & a_{2Z} + a_{4Z} x_1 + 2 a_{5Z} x_2 \end{bmatrix}
\begin{bmatrix} a_{1Z} x_1 + a_{2Z} x_2 \\ b_{0Z} + b_{1Z} x_1 + b_{2Z} x_2 \end{bmatrix} + \cO \left( \| x \|^2 \right) \nonumber \\
&= a_{2Z} b_{0Z} + \left( a_{1Z}^2 + a_{2Z} b_{1Z} + a_{4Z} b_{0Z} \right) x_1
+ \left( a_{1Z} a_{2Z} + a_{2Z} b_{2Z} + 2 a_{5Z} b_{0Z} \right) x_2 + \cO \left( \| x \|^2 \right), \nonumber \\
\cL_{f^Z}^3 H(x) &= \nabla \left( \cL_{f^Z}^2 H(x) \right)^{\sf T} f^Z(x) \nonumber \\
&= \begin{bmatrix} a_{1Z}^2 + a_{2Z} b_{1Z} + a_{4Z} b_{0Z} & a_{1Z} a_{2Z} + a_{2Z} b_{2Z} + 2 a_{5Z} b_{0Z} \end{bmatrix}
\begin{bmatrix} 0 \\ b_{0Z} \end{bmatrix} + \cO \left( \| x \| \right) \nonumber \\
&= \left( a_{1Z} a_{2Z} + a_{2Z} b_{2Z} + 2 a_{5Z} b_{0Z} \right) b_{0Z} + \cO \left( \| x \| \right). \nonumber
\end{align}
By evaluating these at the origin and substituting them into \eqref{eq:Lambda}, we obtain
\begin{equation}
\Lambda(0) = \frac{1}{a_{2L}} \left( \frac{a_{1L}}{b_{0L}} + \frac{b_{2L}}{b_{0L}} + \frac{2 a_{5L}}{a_{2L}} \right)
- \frac{1}{a_{2R}} \left( \frac{a_{1R}}{b_{0R}} + \frac{b_{2R}}{b_{0R}} + \frac{2 a_{5R}}{a_{2R}} \right).
\label{eq:Lambda2d1}
\end{equation}
But if \eqref{eq:f} is refracting, then
$\cL_{f^L} H(x) = \cL_{f^R} H(x)$ at any $x = (x_1,x_2)$ with $x_1 = 0$ and $x_2 \in \mathbb{R}$.
This implies $a_{2L} = a_{2R}$ and $a_{5L} = a_{5R}$, in which case \eqref{eq:Lambda2d1} reduces to
\begin{equation}
\Lambda(0) = \frac{1}{a_2} \left( \frac{a_{1L}}{b_{0L}} + \frac{b_{2L}}{b_{0L}}
- \frac{a_{1R}}{b_{0R}} - \frac{b_{2R}}{b_{0R}} \right),
\nonumber
\end{equation}
where $a_2 = a_{2L} = a_{2R}$.
Moreover,
\begin{equation}
\Lambda(0) = \frac{\sigma_L - \sigma_R}{a_2},
\nonumber
\end{equation}
by \eqref{eq:2dsigmaZ}.
Since $a_2 > 0$ (by the invisibility assumption),
the sign of $\Lambda$ is equivalent to the sign of $\sigma_L - \sigma_R$.

\section{Relation to second-order sliding mode control}
\label{sec:control}

Many engineered systems contain control algorithms designed to impose certain dynamical behaviours.
The control strategies commonly contain switches, because these are simple to implement and often mathematically optimal \cite{AtFa66}.
The switches introduce discontinuities, and {\em sliding mode control}
refers to control systems that are designed to maintain the system state on a discontinuity threshold.
The idealised description of this behaviour is sliding motion on an attracting sliding region \cite{JoRa99}.

But with a hard switch the system in practice usually switches rapidly and repeatedly,
rather than settling to smooth averaged motion.
Such {\em chattering} is undesirable as it strains physical components leading to wear and tear.
There are many modifications of hard switches that mitigate chattering,
one of them being the addition of an integrator
that leads to {\em second-order sliding mode control} \cite{ShEd14,BaPi03,BoFr07}.
For example, suppose we wish to bring the value of a quantity $x_1 \in \mathbb{R}$ to zero.
Instead of direct control $\dot{x}_1 = u$, for which the sliding surface is $x_1 = 0$,
with an integrator we have
\begin{equation}
\ddot{x}_1 = u,
\label{eq:Bo00f}
\end{equation}
and the second-order sliding surface is $x_1 = \dot{x}_1 = 0$.
Various algorithms have been developed for the control action $u$,
such as the {\em twisting algorithm} and {\em sub-optimal algorithm} \cite{BaPi03,FrLe02,Le03b}.
These reduce chattering because the idealised system contains no sliding regions.

However, the algorithms typically employ multiple switches,
in which case the form of the idealised system is more complicated than \eqref{eq:f}.
This enables the system to achieve the control task in finite time,
which does not occur in our setting by Theorem \ref{th:noZeno}.

For instance, suppose we wish to minimise $\int_0^\infty x_1(t)^2 \,dt$
subject to the constraint $|u| \le 1$.
If the system has the form \eqref{eq:Bo00f}, the optimal control strategy is
\begin{equation}
u = \begin{cases}
1, & x_1 < -C x_2 |x_2|, \\
-1, & x_1 > -C x_2 |x_2|,
\end{cases}
\label{eq:Bo00u}
\end{equation}
where $x_2 = \dot{x}_1$ and $C = \sqrt{\frac{\sqrt{33} - 1}{24}} \approx 0.4446$, see Borisov \cite{Bo00}.
This system can be written in the form \eqref{eq:f}
using $H(x) = x_1 + C x_2 |x_2|$, see Fig.~\ref{fig:ppFuller}.
However, Theorem \ref{th:noZeno} does not apply to this system because $H$ is not $C^1$.

\begin{figure}
\begin{center}
\includegraphics[width=8cm]{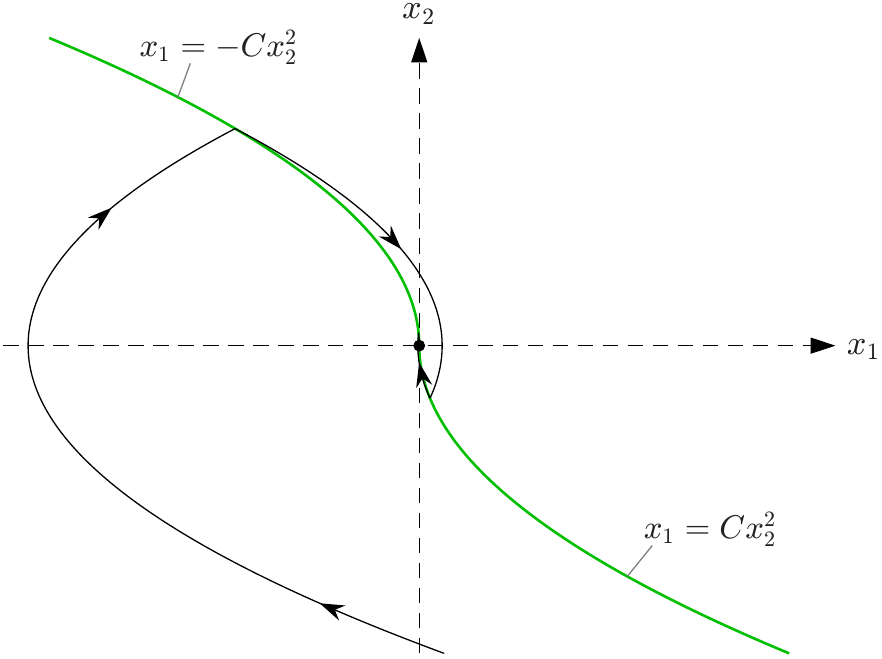}
\caption{
A typical orbit of the control system \eqref{eq:Bo00f}--\eqref{eq:Bo00u},
where $x_2 = \dot{x}_1$ and $C = \sqrt{\frac{\sqrt{33} - 1}{24}}$.
This illustrates Fuller's phenomenon \cite{Bo00}
whereby the control task is completed in finite time through infinitely many switching events.
\label{fig:ppFuller}
} 
\end{center}
\end{figure}

Indeed, it is a simple exercise to show
that the forward orbit of any $x = (-C \zeta_0^2,\zeta_0)$ with $\zeta_0 > 0$
next returns to the $x_2 > 0$ part of the discontinuity surface at
$x = (-C \zeta_1^2,\zeta_1)$, where $\zeta_1 = \left( 1 - \frac{4 C}{2 C + 1} \right) \zeta_0$, after a time
\begin{equation}
\tau = \left( 1 + \sqrt{1 - \frac{4 C}{2 C + 1}} \right)^2 \zeta_0.
\label{eq:Bo00tau}
\end{equation}
Consequently, subsequent switching times obey a geometric series and the orbit
reaches the origin at the time
\begin{align}
\tau_{\rm total} &= \left( 1 + \sqrt{1 - \frac{4 C}{2 C + 1}} \right)^2
\sum_{j=0}^\infty \left( 1 - \frac{4 C}{2 C + 1} \right)^j \zeta_0 \nonumber \\
&= \frac{2 C + 1}{4 C} \left( 1 + \sqrt{1 - \frac{4 C}{2 C + 1}} \right)^2 \zeta_0 \,.
\end{align}
This is an example of Zeno's phenomenon of infinitely many switches in finite time,
also known as Fuller's phenomenon in this context \cite{Bo00}.

In summary, applied examples of second-order sliding mode control systems
do not fit the framework developed in this paper.
This is because refracting systems
do not allow for infinitely many switches in finite time,
which is a desirable feature for idealised control systems that demand the control task be completed promptly.

\section{Discussion}
\label{sec:conc}

In this paper we have established basic mathematical properties for refracting Filippov systems
which are characterised by continuity of the first
Lie derivative across discontinuity surfaces (Definition \ref{df:secondOrder}).
The most significant feature of refracting systems
is spiralling motion about an invisible-invisible subset of a codimension-two tangency surface.
This occurs for the block-damper system of Fig.~\ref{fig:BlockDamper} when $k < \cA \cos(x_3) < k + k_D d$.
In this case the external forcing overcomes the force of the spring,
but not the combined force of the spring and the damper.
Consequently the block is repeatedly pushed onto and off the damper,
corresponding to spiralling motion in phase space, Fig.~\ref{fig:Imp}.
Second-order sliding motion corresponds to the situation that
the block is touching but not displacing the damper for a sustained period of time.

If a spiralling orbit remains near an attracting invisible-invisible region (see Definition \ref{df:attrRep}),
then it converges to the invisible-invisible region.
By Theorems \ref{th:consistency} and \ref{th:noZeno}, the evolution of such orbits
tends to second-order sliding motion as $t \to \infty$.
Since the absence of sliding regions is a natural consequence of physical assumptions,
we anticipate that this behaviour can be found in models of diverse areas of application.

One could in future analyse $r^{\rm th}$-order Filippov systems
characterised by having continuous $j^{\rm th}$ Lie derivatives for all $j < r$.
However, in view of theoretical control theory for piecewise-linear systems \cite{FrLe02},
we expect that $r^{\rm th}$-order sliding motion on a codimension-$r$ subset of a discontinuity surface
is always unstable in the sense that from any point on the subset there exists
a Filippov solution that immediately departs the subset.

This paper is intended as an essential step toward a general theory of refracting systems.
We have characterised the stability of equilibria of the second-order sliding vector field $f^T$ through the quantity $\Lambda$ \eqref{eq:Lambda},
but it remains to characterise the stability of other invariant sets of $f^T$.
For example, if a periodic orbit of $f^T$ is asymptotically stable in regards to second-order sliding motion,
we expect that nearby spiralling orbits converge to it if the average value of $\Lambda$ over the periodic orbit is negative.

Also it remains to analyse boundary-equilibrium bifurcations of refracting systems.
For a generic Filippov system the collision of an equilibrium with a discontinuity surface
is always associated with a pseudo-equilibrium, and the two equilibria are either admissible
on the same side of the bifurcation (nonsmooth fold) or on opposite sides of the bifurcation (persistence).
We expect this occurs analogously for refracting systems,
but involving instead second-order pseudo-equilibria.
In the case of the ant colony migration model of \S\ref{sec:ants},
we expect such bifurcations are responsible for the existence of the second-order pseudo-equilibria $x^T$ of Fig.~\ref{fig:Ant}.

\section*{Acknowledgements}

This work was supported by Marsden Fund contract MAU2504 managed by Royal Society Te Ap\={a}rangi.
The author thanks Douglas Novaes and Tiago Siller for
bringing to their attention several references on refracting systems.

\appendix

\section{A derivation of \eqref{eq:dVR}}
\label{app:dVR}

Here we show that the identity $\nabla V(x)^{\sf T} f^R(x) = \cL_{f^R}^2 H(x)$, given above as \eqref{eq:dVR},
is valid for any $x \in \cT$.
Since $\cL_{f^L} H(x) = \cL_{f^R} H(x)$ for all $x \in \Omega$ for which $H(x) = 0$,
there exists a scalar field $\alpha$ such that $\cL_{f^L} H(x) = \cL_{f^R} H(x) + \alpha(x) H(x)$ for all $x \in \Omega$.
Then since $V(x) = \cL_{f^L} H(x)$,
\begin{align*}
\nabla V(x)^{\sf T} f^R(x) &= \left( \nabla \cL_{f^R} H(x) \right)^{\sf T} f^R(x)
+ \alpha(x) \nabla H(x)^{\sf T} f^R(x)
+ H(x) \nabla \alpha(x)^{\sf T} f^R(x),
\end{align*}
and by \eqref{eq:LieDeriv2} this can be written as
\begin{equation}
\nabla V(x)^{\sf T} f^R(x) = \cL_{f^R}^2 H(x) + \alpha(x) \cL_{f^R} H(x) + H(x) \cL_{f^R} \alpha(x).
\nonumber
\end{equation}
But $H(x) = 0$ and $\cL_{f^R} H(x) = 0$ on $\cT$, thus $\nabla V(x)^{\sf T} f^R(x) = \cL_{f^R}^2 H(x)$.

\section{Calculations for crossing dynamics}
\label{app:crossing}

Here we prove Theorem \ref{th:P} by composing asymptotic formulas associated with the maps $P_L$ and $P_R$.
These formulas are given in Lemmas \ref{le:PR} and Lemma \ref{le:PL}.
Lemma \ref{le:PR} is proved below; Lemma \ref{le:PL} is identical with $L$'s in place of $R$'s,
so follows trivially from Lemma \ref{le:PR}.

\begin{lemma}
Consider a piecewise-$C^3$ system \eqref{eq:f}
and let $S \subset \cT_R$ be a compact set of invisible folds of $f^R$.
There exists a neighbourhood $\cN \subset \Omega$ of $S$
such that $P_R$ is well-defined for all $x \in \cN \cap \Sigma$ with $\cL_{f^R} H(x) = \nu > 0$, and
\begin{align}
\tau_R(x) &= -\frac{2}{\cL_{f^R}^2 H(x)} \,\nu + \cO \left( \nu^2 \right), \label{eq:TR} \\
P_R(x) &= x - \frac{2 f^R(x)}{\cL_{f^R}^2 H(x)} \,\nu + \cO \left( \nu^2 \right), \label{eq:PR} \\
\left( \cL_{f^R} H \right)(P_R(x)) &= -\nu + \frac{2 \cL_{f^R}^3 H(x)}{3 \big( \cL_{f^R}^2 H(x) \big)^2} \,\nu^2 + \cO \left( \nu^3 \right). \label{eq:VPR}
\end{align}
\label{le:PR}
\end{lemma}

\begin{lemma}
Consider a piecewise-$C^3$ system \eqref{eq:f}
and let $S \subset \cT_L$ be a compact set of invisible folds of $f^L$.
There exists a neighbourhood $\cN \subset \Omega$ of $S$
such that $P_L$ is well-defined for all $x \in \cN \cap \Sigma$ with $\cL_{f^L} H(x) = \nu < 0$, and
\begin{align}
\tau_L(x) &= -\frac{2}{\cL_{f^R}^2 H(x)} \,\nu + \cO \left( \nu^2 \right), \label{eq:TL} \\
P_L(x) &= x - \frac{2 f^L(x)}{\cL_{f^L}^2 H(x)} \,\nu + \cO \left( \nu^2 \right), \label{eq:PL} \\
\left( \cL_{f^L} H \right)(P_L(x)) &= -\nu + \frac{2 \cL_{f^L}^3 H(x)}{3 \big( \cL_{f^L}^2 H(x) \big)^2} \,\nu^2 + \cO \left( \nu^3 \right). \label{eq:VPL}
\end{align}
\label{le:PL}
\end{lemma}

\begin{proof}[Lemma \ref{le:PR}]
Let $\varphi_t(x)$ denote the flow induced by $f^R$.
For any $x \in \Sigma$, we have $H(x) = 0$,
thus, by the definition \eqref{eq:LieDeriv} of a Lie derivative,
\begin{equation}
H \left( \varphi_t(x) \right) = \nu t + \frac{1}{2} \cL^2_{f^R} H(x) t^2 + \frac{1}{6} \cL^3_{f^R} H(x) t^3 + \cO \left( t^4 \right),
\label{eq:PLproof1}
\end{equation}
where $\nu = \cL_{f^R} H(x)$.
The error term in \eqref{eq:PLproof1} is justified for the following reason.
The vector field $f^R$ is $C^3$, so the flow $\varphi_t(x)$ is $C^3$ with respect to $x$ and $t$,
but also four times continuously differentiable with respect to $t$
because $\dot{\varphi}_t(x) = f^R \left( \varphi_t(x) \right)$ is $C^3$, see \cite{Ha02,Me17}.

Since $H \left( \varphi_0(x) \right) = H(x) = 0$ on $\Sigma$, the function
$J(t,x) = \frac{H \left( \varphi_t(x) \right)}{t}$ admits a $C^3$ extension
to a neighbourhood of $t = 0$.
Moreover, $J(0,x) = 0$ and $\frac{\partial J}{\partial t}(0,x) = \frac{1}{2} \cL^2_{f^R} H(x) \ne 0$ for all $x \in S$,
so by the implicit function theorem there exists a unique locally-defined $C^3$ function $\tau_R(x)$ such that
$J \left( \tau_R(x), x \right) = 0$.
By \eqref{eq:PLproof1} we obtain
\begin{equation}
\tau_R(x) = -\frac{2}{\cL^2_{f^R} H(x)} \,\nu - \frac{4 \cL^3_{f^R} H(x)}{3 \cL^2_{f^R} H(x)^3} \,\nu^2
+ \cO \left( \nu^3 \right),
\label{eq:PLproof10}
\end{equation}
which verifies \eqref{eq:TR}.
By substituting \eqref{eq:PLproof10}
into $\varphi_t(x) = x + t f^R(x) + \cO \left( t^2 \right)$, we obtain \eqref{eq:PR}.
To evaluate
$\left( \cL_{f^R} H \right)(P_R(x)) = \frac{\partial}{\partial t} H \left( \varphi_t(x) \right) \big|_{t = \tau_R(x)}$,
we differentiate \eqref{eq:PLproof1}, the result being
\begin{equation}
\left( \cL_{f^R} H \right)(P_R(x)) = \nu + \cL^2_{f^R} H(x) \tau_R(x)
+ \frac{1}{2} \cL^3_{f^R} H(x) \tau_R(x)^2 + \cO \left( \nu^3 \right).
\nonumber
\end{equation}
Into this we substitute \eqref{eq:PLproof10}, yielding \eqref{eq:VPR}.
\end{proof}

\begin{proof}[Theorem \ref{th:P}]
For the formulas \eqref{eq:TL}--\eqref{eq:VPL} we use \eqref{eq:VPR} in place of $\nu$,
because this is the value of $V$ at the start of the orbit that follows $f^L$,
while the formulas \eqref{eq:TR}--\eqref{eq:VPR} are each used as is.
By summing \eqref{eq:TR} and \eqref{eq:TL} we obtain \eqref{eq:T},
by composing \eqref{eq:PR} and \eqref{eq:PL} we obtain \eqref{eq:P},
and by evaluating \eqref{eq:VPL} with \eqref{eq:VPR} we obtain \eqref{eq:VP}.
\end{proof}

\section{Proof of Theorem \ref{th:consistency}}
\label{app:consistency}

\begin{proof}[Theorem \ref{th:consistency}]
Let
\begin{equation}
\Sigma^+ = \left\{ x \in \Sigma \,\middle|\, V(x) > 0 \right\},
\label{eq:SigmaPlus}
\end{equation}
and write $B_r(z)$ for the open ball of radius $r > 0$ centred at $z \in \mathbb{R}^n$.
By Theorem \ref{th:P}, we have for $x \in \Sigma^+$ near $S$
\begin{align*}
\tau(x) &= \beta(x) V(x) + \cO \left( V(x)^2 \right), \\
P(x) &= x + \beta(x) f^T(x) V(x) + \cO \left( V(x)^2 \right), \\
V(P(x)) &= V(x) + \cO \left( V(x)^2 \right),
\end{align*}
where $\beta$ is given by \eqref{eq:beta}.
Thus there exist $\beta_0, \beta_1, \delta_1, M_1 > 0$ such that
for all $x$ belonging to the set
\begin{equation}
\Sigma^+_{\delta_1} = \left\{ x \in \Sigma^+ \,\middle|\, \text{$\| x - y \| < \delta_1$ for some $y \in S$} \right\},
\nonumber
\end{equation}
$P(x)$ is well-defined and
\begin{align}
\beta_0 V(x) &\le \tau(x) \le \beta_1 V(x), \label{eq:prop1} \\
\left| \tau(x) - \beta(x) V(x) \right| &\le M_1 V(x)^2, \label{eq:prop2} \\
\left\| P(x) - \left( x + f^T(x) \beta(x) V(x) \right) \right\| &\le M_1 V(x)^2, \label{eq:prop3} \\
\left| V(P(x)) - V(x) \right| &\le M_1 V(x)^2, \label{eq:prop4} \\
\left\| \nabla V(x) \right\| &\le M_1 \,. \label{eq:LipV}
\end{align}
Let $M_2 > 0$ be such that $\| f^T \| \le M_2$ on $S$
and $\| f^L \|, \| f^R \| \le M_2$ along the orbit of \eqref{eq:f} from any $x \in \Sigma^+_{\delta_1}$ to $P(x)$.
Let $K_1 = M_1 (1 + M_2) + \beta_1^2 M_3$.

Since $f^T$ is $C^2$, there exists $M_3 \ge \frac{M_1}{\beta_0}$ such that for all $y \in S$
\begin{equation}
\left\| \rD f^T(y) \right\| \le M_3 \,,
\label{eq:LipfT}
\end{equation}
and the forward orbit $\xi(t)$ of $y$ under $f^T$ obeys
\begin{equation}
\left\| \xi(t + \Delta t) - \left( \xi(t) + f^T(\xi(t)) \Delta t \right) \right\| \le M_3 (\Delta t)^2,
\label{eq:prop9}
\end{equation}
for all $0 \le \Delta t \le \beta_1 \delta_1$ and $0 \le t \le b - \Delta t$.
To complete the proof we show \eqref{eq:bound} holds for $x \in \Sigma^+$ with
\begin{align}
\delta &= \frac{\delta_1}{1 + \frac{b K_1 M_1}{\beta_0}} \,\re^{-M_3 b}, \label{eq:delta} \\
M &= \left( 1 + \tfrac{b K_1 M_1}{\beta_0} + 2 \beta_1 M_1 M_2 \right) \re^{M_3 b}. \label{eq:M}
\end{align}
The result is easily extended to points $x \in \Omega \setminus \cT$ with $x \notin \Sigma^+$
by decreasing the value of $\delta$ if necessary, because the forward orbits of such points quickly reach $\Sigma^+$.

Choose any $y \in S_0$ and $x \in \Sigma^+$ with $\| x - y \| < \delta$.
Let $\phi(t)$ be the forward orbit of $x$ under \eqref{eq:f},
and $\xi(t)$ be the forward orbit of $y$ under $f^T$.
Let $0 = t_0 < t_1 < \cdots < t_N$ be the times at which $\phi(t) \in \Sigma^+$,
with $N$ such that $t_{N-1} < b \le t_N$.
Let $x^{(j)} = \phi(t_j)$ be the corresponding points on $\Sigma^+$.
Notice $x^{(j+1)} = P \left( x^{(j)} \right)$
and $t_{j+1} = t_j + \tau \left( x^{(j)} \right)$
for all $j = 0,1,\ldots,N-1$.

Write $E(t) = \left\| \phi(t) - \xi(t) \right\|$.
We now show
\begin{align}
x^{(j)} &\in \Sigma^+_{\delta_1}, \label{eq:induct1} \\
V \left( x^{(j)} \right) &\le M_1 \re^{\frac{M_1 t_j}{\beta_0}} \| x - y \|, \label{eq:induct2} \\
E(t_j) &\le \left( 1 + \tfrac{K_1 M_1 t_j}{\beta_0} \right) \re^{M_3 t_j} \| x - y \|, \label{eq:induct3}
\end{align}
for all $j = 0,1,\ldots,N-1$.
This is achieved by induction on $j$.
The result is true for $j = 0$ because $\delta < \delta_1$, so \eqref{eq:induct1} holds,
and then $V \left( x^{(0)} \right) = V(x) \le M_1 \| x - y \|$ by \eqref{eq:LipV}, so \eqref{eq:induct2} holds,
and finally $E(t_0) = E(0) = \| x - y \|$, so \eqref{eq:induct3} holds with equality.

Suppose \eqref{eq:induct1}--\eqref{eq:induct3} hold for some $j = 0,1,\ldots,N-2$.
By \eqref{eq:prop4} and \eqref{eq:induct2}
\begin{align}
V \left( x^{(j+1)} \right) &\le V \left( x^{(j)} \right) + M_1 V \left( x^{(j)} \right)^2 \nonumber \\
&\le \left( 1 + M_1 V \left( x^{(j)} \right) \right) M_1 \re^{\frac{M_1 t_j}{\beta_0}} \| x - y \|.
\label{eq:VboundMid}
\end{align}
Observe
\begin{equation}
1 + M_1 V \left( x^{(j)} \right)
\le \re^{M_1 V \left( x^{(j)} \right)}
\le \re^{\frac{M_1}{\beta_0} \tau \left( x^{(j)} \right)},
\nonumber
\end{equation}
using \eqref{eq:prop1}.
So since $\tau \left( x^{(j)} \right) = t_{j+1} - t_j$,
\eqref{eq:VboundMid} implies \eqref{eq:induct2} with $j+1$.

To verify \eqref{eq:induct3} with $j+1$, observe
\begin{align}
E(t_{j+1}) &= \left\| P \left( x^{(j)} \right) - \xi(t_{j+1}) \right\| \nonumber \\
&\le \left\| P \left( x^{(j)} \right) - \Big( x^{(j)}
+ \beta \left( x^{(j)} \right) f^T \left( x^{(j)} \right) V \left( x^{(j)} \right) \Big) \right\| \nonumber \\
&\quad+ \left\| x^{(j)} + \beta \left( x^{(j)} \right) f^T \left( x^{(j)} \right) V \left( x^{(j)} \right)
- \Big( \xi(t_j) + f^T \left( \xi(t_j) \right) \tau \left( x^{(j)} \right) \Big) \right\| \nonumber \\
&\quad+ \left\| \xi(t_j) + f^T \left( \xi(t_j) \right) \tau \left( x^{(j)} \right) - \xi(t_{j+1}) \right\|.
\label{eq:conProof30}
\end{align}
By \eqref{eq:prop3}, the first term in \eqref{eq:conProof30}
is bounded by $M_1 V \left( x^{(j)} \right)^2$.
By \eqref{eq:prop1}, $\tau \left( x^{(j)} \right) \le \beta_1 V \left( x^{(j)} \right)$,
and so by \eqref{eq:prop9} the third term in \eqref{eq:conProof30}
is bounded by $\beta_1^2 M_3 V \left( x^{(j)} \right)^2$.
The middle term in \eqref{eq:conProof30} is bounded by
\begin{align}
& E(t_j) + \left\| f^T \left( x^{(j)} \right) \right\|
\left\| \beta \left( x^{(j)} \right) V \left( x^{(j)} \right) - \tau \left( x^{(j)} \right) \right\|
+ \left\| f^T \left( x^{(j)} \right) - f^T \left( \xi(t_j) \right) \right\| \left| \tau \left( x^{(j)} \right) \right| \nonumber \\
&\le E(t_j) + M_1 M_2 V \left( x^{(j)} \right)^2 + M_3 (t_{j+1} - t_j) E(t_j),
\end{align}
where we have used \eqref{eq:prop2}, \eqref{eq:LipfT}, the bound on $\| f^T \|$,
and substituted $\tau \left( x^{(j)} \right) = t_{j+1} - t_j$.
By inserting these bounds into \eqref{eq:conProof30}, we obtain
\begin{equation}
E(t_{j+1}) \le \left( 1 + M_3 (t_{j+1} - t_j) \right) E(t_j) + K_1 V \left( x^{(j)} \right)^2,
\label{eq:conProof40}
\end{equation}
using also the definition of $K_1$.
Into \eqref{eq:conProof40} we substitute
$1 + M_3 (t_{j+1} - t_j) \le \re^{M_3 (t_{j+1} - t_j)}$,
and $V \left( x^{(j)} \right) \le \frac{1}{\beta_0} \tau \left( x^{(j)} \right) = \frac{t_{j+1} - t_j}{\beta_0}$,
which uses \eqref{eq:prop1}.
We also substitute \eqref{eq:induct2}, resulting in
\begin{equation}
E(t_{j+1}) \le \re^{M_3 (t_{j+1} - t_j)} E(t_j) + \tfrac{K_1 M_1 (t_{j+1} - t_j)}{\beta_0}
\,\re^{\frac{M_1 t_j}{\beta_0}} \| x - y \|.
\nonumber
\end{equation}
By then inserting \eqref{eq:induct3}
and the bounds $M_3 \ge \frac{M_1}{\beta_0}$ and $t_{j+1} > t_j$ in the exponent,
we obtain \eqref{eq:induct3} with $j+1$.
This implies $E(t_{j+1}) < \delta_1$, due to the choice \eqref{eq:delta} and because $t_{j+1} < b$,
thus \eqref{eq:induct1} holds with $j+1$.
This completes our inductive verification of \eqref{eq:induct1}--\eqref{eq:induct3} for all $j = 0,1,\ldots,N-1$.

Finally, choose any $t \in [0,b]$,
and let $j \in \{ 0,1,\ldots,N-1 \}$ be such that $t_j \le t < t_{j+1}$.
By \eqref{eq:prop1} and the bound on $\| f^L \|$, $\| f^R \|$, and $\| f^T \|$,
\begin{align}
E(t) &= \left\| \phi(t) - \xi(t) \right\| \nonumber \\
&\le \left\| \phi(t) - \phi(t_j) \right\|
+ \left\| \phi(t_j) - \xi(t_j) \right\|
+ \left\| \xi(t_j) - \xi(t) \right\| \nonumber \\
&\le \beta_1 M_2 V \left( x^{(j)} \right) + E(t_j) + \beta_1 M_2 V \left( x^{(j)} \right).
\nonumber
\end{align}
Then by \eqref{eq:induct2} and \eqref{eq:induct3}
\begin{align}
E(t) &\le \left( 1 + \frac{K_1 M_1 t_j}{\beta_0} \right) \re^{M_3 t_j} \| x - y \|
+ 2 \beta_1 M_1 M_2 \re^{\frac{M_1 t_j}{\beta_0}} \| x - y \| \nonumber \\
&\le \left( 1 + \frac{b K_1 M_1}{\beta_0} + 2 \beta_1 M_1 M_2 \right) \re^{M_3 b} \| x - y \|
\nonumber
\end{align}
using $M_3 \ge \frac{M_1}{\beta_0}$ and $t_j \le b$ for the second inequality.
By \eqref{eq:M} this verifies \eqref{eq:bound} as required.
\end{proof}

\section{Lyapunov stability}
\label{app:LyapStable}

Here write $\varphi_t(x)$ for the flow induced by \eqref{eq:f} on $\Omega \setminus \cT$,
and $\psi_t(y)$ for the flow induced by $f^T$ on $\cT$.
As above, $\Sigma^+$ denotes the set \eqref{eq:SigmaPlus}, and $B_r(z)$ is an open ball.

\begin{lemma}
For \eqref{eq:f} satisfying the assumptions of case (1) of Theorem \ref{th:stab2}, $x^*$ is Lyapunov stable.
\label{le:noEscape}
\end{lemma}

\begin{proof}
Choose any $\ee > 0$.
We need to show there exists $\delta > 0$ such that the forward orbit of any $x \in B_\delta(x^*)$ does not escape $B_\ee(x^*)$.
If $x \in \cT$, then evolution is under $f^T$
and stays near $x^*$ because all eigenvalues of $\rD \tilde{f}^T(\tilde{x}^*)$ have negative real-part,
so it suffices to construct $\delta$ to handle $x \notin \cT$.

By the eigenvalue assumption, $x^*$ is exponentially stable for $f^T$.
That is, there exist $\alpha_1 \ge 1$, $\beta_1 > 0$, $\delta_1 > 0$ such that
if $y \in B_{\delta_1}(x^*) \cap \cT$, then
\begin{equation}
\| \psi_t(y) - x^* \| \le \alpha_1 \re^{-\beta_1 t} \| y - x^* \|, \qquad \text{for all $t \ge 0$}.
\label{eq:noEscapeProof1}
\end{equation}
Since $x^*$ belongs to an attracting invisible-invisible region, we have
$\beta(x^*) > 0$ and $\Lambda(x^*) < 0$, so by \eqref{eq:T} and \eqref{eq:VP}
we can assume $\delta_1 > 0$ has been chosen small enough that
\begin{align}
\tau(x) &\le 2 \beta(x^*) V(x), \label{eq:noEscapeProof2a} \\
V(P(x)) &\le V(x), \label{eq:noEscapeProof2b} 
\end{align}
for all $x \in B_{\delta_1 \alpha_1}(x^*) \cap \Sigma^+$.
We can also assume $\delta_1$ has been chosen small enough that
the closure of $B_{\delta_1 \alpha_1}(x^*) \cap \cT$, call it $S$, is an invisible-invisible region.
Let $S_0 = B_{\delta_1}(x^*) \cap \cT$
and observe from \eqref{eq:noEscapeProof1} that $\psi_t(y) \in S$ for all $y \in S_0$ and $t \ge 0$.
Let $b > 0$ be such that
\begin{equation}
\alpha_1 \re^{-\frac{\beta_1 b}{2}} = \tfrac{1}{4}.
\label{eq:noEscapeProof3}
\end{equation}
By Theorem \ref{th:consistency} there exists $\delta_2 > 0$ and $M \ge \frac{1}{2}$
such that for all $y \in S_0$ and $x \in \Omega \setminus \cT$ with $\| x - y \| < \delta_2$ we have
\begin{equation}
\| \varphi_t(x) - \psi_t(y) \| \le M \| x - y \|, \qquad \text{for all $0 \le t \le b$}.
\label{eq:noEscapeProof10}
\end{equation}
Let $K = \| \nabla V(x^*) \|$, where $K > 0$ because $\cL_{f^L}^2 H(x^*) \ne 0$.
Since $x^* \in \cT$ and $\cT$ is the set of all $x \in \Sigma$ for which $V(x) = 0$,
there exists $\delta_3 > 0$ such that for all $x \in B_{\delta_3}(x^*) \cap \Sigma^+$ we have
\begin{align}
V(x) &\le 2 K \| x - x^* \|, \label{eq:noEscapeProof13} \\
\| x - y \| &\le \tfrac{2}{K} V(x), \qquad \text{for some $y \in \cT$}. \label{eq:noEscapeProof14}
\end{align}

Assume $\ee \le {\rm min}\left[2 \alpha_1 \delta_1, 8 \alpha_1 M \delta_2, 4 \alpha_1 \delta_3, \frac{4 \alpha_1 M b}{K \beta(x^*)} \right]$.
Let $\delta > 0$ be such that the forward orbit
of any $x \in B_\delta(x^*) \setminus \cT$
first hits $\Sigma^+$ at a point in $B_{\frac{\ee}{32 \alpha_1 M}}(x^*)$.
It remains to show that the forward orbit of any point in
\begin{equation}
U = \left\{ x \in B_{\frac{\ee}{4 \alpha_1}}(x^*) \cap \Sigma^+ \,\middle|\, V(x) \le \tfrac{K \ee}{16 \alpha_1 M} \right\},
\nonumber
\end{equation}
returns to $U$ in a time $\frac{b}{2} \le \tilde{t} \le b$ without escaping $B_\ee(x^*)$.
This is because by \eqref{eq:noEscapeProof13} it implies that the forward orbit of any $x \in B_\delta(x^*) \setminus \cT$
reaches $U$, then by recursion remains in $B_\ee(x^*)$ for all $t \ge 0$.

Choose any $x \in U$.
By \eqref{eq:noEscapeProof14}, there exists $y \in \cT$ with $\| x - y \| \le \frac{\ee}{8 \alpha_1 M}$, and notice
\begin{align}
\| y - x^* \| &\le \| y - x \| + \| x - x^* \| \nonumber \\
&\le \frac{\ee}{8 \alpha_1 M} + \frac{\ee}{4 \alpha_1} \nonumber \\
&\le \frac{\ee}{2 \alpha_1}, \nonumber
\end{align}
since $M \ge \frac{1}{2}$.
By \eqref{eq:noEscapeProof1} and \eqref{eq:noEscapeProof10} we have for all $0 \le t \le b$,
\begin{align}
\left\| \varphi_t(x) - x^* \right\| &\le \| \varphi_t(x) - \psi_t(y) \| + \| \psi_t(y) - x^* \| \nonumber \\
&\le M \times \tfrac{\ee}{8 \alpha_1 M} + \alpha_1 \times \tfrac{\ee}{2 \alpha_1}, \nonumber \\
&\le \ee, \nonumber
\end{align}
since $\alpha_1 \ge 1 > \frac{1}{4}$.
The orbit $\varphi_t(x)$ repeatedly intersects $\Sigma^+$, but always with $V < \tfrac{K \ee}{16 \alpha_1 M}$ by \eqref{eq:noEscapeProof2b}.
By \eqref{eq:noEscapeProof2a}, the time between consecutive intersections is at most $\frac{b}{2}$,
thus there exists $\frac{b}{2} \le \tilde{t} \le b$ such that $\varphi_{\tilde{t}}(x) \in \Sigma^+$.
Then, using also \eqref{eq:noEscapeProof3},
\begin{align}
\left\| \varphi_{\tilde{t}}(x) - x^* \right\| &\le \| \varphi_{\tilde{t}}(x) - \psi_{\tilde{t}}(y) \| + \| \psi_{\tilde{t}}(y) - x^* \| \nonumber \\
&\le M \times \tfrac{\ee}{8 \alpha_1 M} + \tfrac{1}{4} \times \tfrac{\ee}{2 \alpha_1}, \nonumber \\
&\le \tfrac{\ee}{4 \alpha_1}, \nonumber
\end{align}
and so $\varphi_{\tilde{t}}(x) \in U$, as required.
\end{proof}

{\footnotesize
\bibliographystyle{unsrt}
\bibliography{SOSIntro2arXivBib}
}

\end{document}